\documentclass{article}
\usepackage{amsfonts}
\usepackage{amsmath}
\usepackage{graphicx}
\usepackage{float}
\usepackage{epstopdf}

\newtheorem{theorem}{Theorem}

\newtheorem{corollary}[theorem]{Corollary}

\newtheorem{definition}[theorem]{Definition}

\newtheorem{lemma}[theorem]{Lemma}

\newtheorem{proposition}[theorem]{Proposition}
\newtheorem{remark}[theorem]{Remark}

\newenvironment{proof}[1][Proof]{\noindent\textbf{#1.} }{\ \rule{0.5em}{0.5em}}

\begin{document}

\title{Singular Infinite Horizon Quadratic Control of Linear Systems with
Known Disturbances: \\
A Regularization Approach}
\author{Valery Y. Glizer \\
ORT Braude College of Engineering, Karmiel, Israel\\
valery48@braude.ac.il \and Oleg Kelis \\
ORT Braude College of Engineering, Karmiel Israel\\
Haifa University, Haifa, Israel,\\
olegkelis@braude.ac.il}
\maketitle

\begin{abstract}
An optimal control problem with an infinite horizon quadratic cost
functional for a linear system with a known additive disturbance is
considered. The feature of this problem is that a weight matrix of the
control cost in the cost functional is singular. Due to this singularity,
the problem can be solved neither by application of the Pontriagin's Maximum
Principle, nor using the Hamilton-Jacobi-Bellman equation approach, i.e.
this problem is singular. Since the weight matrix of the control cost, being
singular, is not in general zero, only a part of the control coordinates is
singular, while the others are regular. This problem is solved by a
regularization method. Namely, it is associated with a new optimal control
problem for the same equation of dynamics. The cost functional in this new
problem is the sum of the original cost functional and an infinite horizon
integral of the squares of the singular control coordinates with a small
positive weight. Due to a smallness of this coefficient, the new problem is
a partial cheap control problem. Using a perturbation technique, an
asymptotic analysis of this partial cheap control problem is carried out.
Based on this analysis, the infimum of the cost functional in the original
problem is derived, and a minimizing sequence of state-feedback controls is
designed. An illustrative example of a singular trajectory tracking is
presented.
\end{abstract}

\section{Introduction}

In this paper, an optimal control of a linear differential equation with
constant coefficients for the state and the control, and with a known
additive disturbance is considered. The control process is evaluated by an
infinite horizon quadratic cost functional to be minimized by a proper
choice of the control. A weight matrix of the control cost in this cost
functional is singular, meaning that the optimal control problem is
singular. Namely, the considered problem can be solved neither by
application of the Pontriagin's Maximum Principle \cite%
{pontr-bolt-gakr-misch-86}, nor using the Hamilton-Jacobi-Bellman equation
approach (Dynamic Programming approach) \cite{bel-57}. This occurs because
the problem of maximization of the corresponding variational Hamiltonian
with respect to the control either has no solution, or has infinitely many
solutions. To the best of our knowledge, five main methods of solution of
singular optimal control problems can be distinguished in the literature.
Thus, higher order necessary or sufficient optimality conditions can be
helpful in solving a singular optimal control problem (see e.g. \cite%
{kelly-64,bel-jac-75,gab-kir-72,Mehr-89,krotov-96,Ferran-Ntog-2014} and
references therein). However, such conditions fail to yield a candidate
optimal control (an optimal control) for the problem, having no solution (an
optimal control) in the class of regular functions, even if the cost
functional has a finite infimum in this class of functions. The second
method propose to derive a singular optimal control as a minimizing sequence
of regular open-loop controls, i.e., a sequence of regular control functions
of time, along which the cost functional tends to its infimum (see e.g. \cite%
{krotov-96,gur-65,gur-dix-77} and references therein). The derivation of the
minimizing sequence is based on the Krotov's sufficient optimality
conditions \cite{krotov-96} and theory of linear first order partial
differential equations. A generalization of this method is the extension
approach \cite{gur-kang-11-1,gur-kang-11-2,gur-kang-11-3}. The third method
combines geometric and analytic approaches. Namely, this method is based on
a decomposition of the state space into "regular" and "singular" subspaces,
and a design an optimal open loop control as a sum of impulsive and regular
functions (see e.g. \cite{Hau-Silv-83,Wil-Kit-Sil-86,Geerts-89,Geerts-94}
and references therein). The fourth method consists in searching a solution
of a singular optimal control problem in a properly defined class of
generalized functions (see e.g. \cite{zav-ses-97}). Finally, the fifth
method based on a regularization of the original singular problem by a
"small" correction of its "singular" cost functional (see e.g. \cite%
{Glizer2012a,Glizer2012b,Glizer2014} and references therein). Such a
regularization is a kind of the Tikhonov's regularization of ill-posed
problems \cite{Tikh-Ars-77}. This method yields the solution of the original
problem in the form of a minimizing sequence of state-feedback controls.

In the present paper, a singular infinite horizon linear-quadratic optimal
control problem with a known additive time-varying disturbance in the
dynamics is considered. To the best of our knowledge, such a problem has not
yet been considered in the literature. This problem is treated by the
regularization, which yields an auxiliary partial cheap control problem.
Using perturbation techniques, an asymptotic behavior of the solution to the
auxiliary control problem is analyzed. Based on this analysis, the existence
of the finite infimum of the cost functional in the original (singular)
control problem is established. The expression for this infimum is derived.
The minimizing sequence of state-feedback controls in the original problem
also is designed. The theoretical results are applied to solution of a
singular tracking problem.

The paper is organized as follows. In Section 2, the rigorous formulation of
the problem is presented. Objectives of the paper are stated. In Section 3,
a regularization of the original singular optimal control problem is made,
yielding a partial cheap control problem. An asymptotic analysis of this
problem is carried out in Section 4. The solution of the original singular
optimal control problem is derived in Section 5. Section 6 deals with an
illustrative example. Section 7 contains some concluding remarks. Proofs of
some technical lemmas and of one of the main theorems are placed in
Appendices.

\section{Problem Statement}

\subsection{Initial control problem and main assumptions \label{Sec2.1}}

Consider the following controlled differential equation:
\begin{equation}  \label{or-eq}
\frac{dZ(t)}{dt}={\mathcal{A}}Z(t)+{\mathcal{B}}U(t)+{\mathcal{F}}(t),\ \ \
t\geq 0, \ \ \ Z(0)=Z_{0},
\end{equation}
where $Z(t)\in E^{n}$ is the state vector; $U(t)\in E^{r}$, ($r\leq n$) is
the control; ${\mathcal{A}}$ and ${\mathcal{B}}$ are given constant matrices
of corresponding dimensions; ${\mathcal{F}}(t)$, $t\ge 0$ is a given
vector-valued function; $Z_{0}\in E^{n}$ is a given vector; for any integer $%
m>0$, $E^{m}$ denotes the real Euclidean space of the dimension $m$.

The cost functional, to be minimized by $U$, is
\begin{equation}
{\mathcal{J}}\big(U\big)\overset{\triangle}{=} \int\limits_{0}^{+\infty }%
\left[ Z^{T}(t){\mathcal{D}}Z(t)+U^{T}(t)GU(t)\right]dt,  \label{perf-ind}
\end{equation}
where ${\mathcal{D}}$ is a given constant symmetric matrix of corresponding
dimension; the given constant $r\times r$-matrix $G$ has the form
\begin{equation}
G=\mathrm{diag}\big(g_{1},...,g_{q},\underbrace{0,...,0}_{r-q}\big),\ \ \ \
0\leq q<r,  \label{mathcal-G_u}
\end{equation}%
the superscript $"T"$ denotes the transposition.

In what follows, we assume:

(\textbf{A1)}\ The matrix ${\mathcal{B}}$ has full rank $r$;

(\textbf{A2)}\ the matrix ${\mathcal{D}}$ is positive semi-definite (${%
\mathcal{D}} \ge 0$);

(\textbf{A3)}\ $g_{k} > 0$, $k=1,...,q$;

(\textbf{A4)}\ the function ${\mathcal{F}}(t)$ satisfies the inequality $\|{%
\mathcal{F}}(t)\|\le a\exp(-\gamma t)$, $t \ge 0$, where $a>0$ and $\gamma>0$
are some constants.

Since the matrix $G$ is singular, the optimal control problem (\ref{or-eq})-(%
\ref{perf-ind}) is singular.

Consider the set ${\mathcal{P}}$ of all functions $p(w,t):E^{n}\times \left[
0,+\infty \right) \rightarrow E^{r}$, which are measurable w.r.t. $t\geq 0$
for any fixed $w\in E^{n}$ and satisfy the local Lipschitz condition w.r.t. $%
w\in E^{n}$ uniformly in $t\geq 0$.

\begin{definition}
\label{admiss-def-auxil}Let $U(Z,t)$, $(Z,t)\in E^{n}\times \left[ 0,+\infty
\right)$, be a function belonging to the set ${\mathcal{P}}$. The function $%
U(Z,t)$ is called an admissible state-feedback control in the problem (\ref%
{or-eq})-(\ref{perf-ind}) if the following conditions hold: (a) the
initial-value problem (\ref{or-eq}) for $U(t)=U(Z,t)$ has the unique locally
absolutely continuous solution $Z(t)$ on the entire interval $\left[
0,+\infty \right)$; (b) $Z(t)\in L^{2}\left[ 0,+\infty; E^{n}\right]$; (c) $U%
\big(Z(t),t\big) \in L^{2}\left[0,+\infty ; E^{r}\right]$. The set of all
such $U(Z,t)$ is denoted by ${\mathcal{M}}_{U}$.
\end{definition}

Denote
\begin{equation}  \label{inf-mathcal J}
\mathcal{J}^{*} \overset{\triangle}{=} \inf_{U(Z,t)\in\mathcal{M}_{U}}%
\mathcal{J}\big(U(Z,t)\big).
\end{equation}

\begin{remark}
\label{finiteness-mathcal-J^*}Since ${\mathcal{D}}\ge 0$ and $G\ge 0$, then
the infimum (\ref{inf-mathcal J}) is nonnegative. Moreover, if ${\mathcal{M}}%
_{U}\neq\emptyset$, this infimum is finite.
\end{remark}

\begin{definition}
The control sequence $\big\{U_{k}(Z,t)\big\}$, $U_{k}(Z,t)\in \mathcal{M}%
_{U} $, $(k=1,2,...)$, is called minimizing in the problem (\ref{or-eq})-(%
\ref{perf-ind}) if
\begin{equation}
\lim_{k\rightarrow +\infty }\mathcal{J}\big(U_{k}(Z,t)\big) = \mathcal{J}%
^{*}.  \label{minim-seq}
\end{equation}
\end{definition}

If there exists $U^{\ast }(Z,t)\in \mathcal{M}_{U}$, for which
\begin{equation}  \label{opt-minim-strat-def}
\mathcal{J}\big(U^{\ast }(Z,t)\big)=\mathcal{J}^{*},
\end{equation}
this control is called optimal in the problem (\ref{or-eq})-(\ref{perf-ind}%
). In this case there exists a minimizing control sequence, point-wise
convergent to $U^{\ast }(Z,t)$ for a. a. $(Z,t)\in E^{n}\times \left[0 ,
+\infty \right)$.

\subsection{Transformation of the problem (\protect\ref{or-eq})-(\protect\ref%
{perf-ind})\label{Sec2.3}}

Let us partition the matrix ${\mathcal{B}}$ into blocks as
${\mathcal{B}} =\Big( {\mathcal{B}}_{1}, {\mathcal{B}}_{2}\Big)$,
where the blocks ${\mathcal{B}}_{1}$ and ${\mathcal{B}}_{2}$ are of
dimensions $n\times q$ and $n\times(r-q)$, respectively.

We assumed that:

(\textbf{A5)}\ \ $\det \left( {\mathcal{B}}_{2}^{T}{\mathcal{DB}}_{2}\right)
\neq 0$.

By ${\mathcal{B}}_{c}$, we denote a complement matrix to the matrix ${%
\mathcal{B}}$, i.e., the matrix of dimension $n\times (n-r)$, and such that
the block matrix $({\mathcal{B}}_{c},{\mathcal{B}})$ is nonsingular. Hence,
the block matrix $\widetilde{{\mathcal{B}}}_{c}=\left( {\mathcal{B}}_{c},{%
\mathcal{B}}_{1}\right)$ is a complement matrix to ${\mathcal{B}}_{2}$.

Consider the following matrices:
\begin{equation}
{\mathcal{H}}=({\mathcal{B}}_{2}^{T}{\mathcal{D}}{\mathcal{B}}_{2})^{-1}{%
\mathcal{B}}_{2}^{T}{\mathcal{D}}\widetilde{{\mathcal{B}}}_{c},\ \ \ \ {%
\mathcal{L}}=\widetilde{{\mathcal{B}}}_{c}-{\mathcal{B}}_{2}{\mathcal{H}}.
\label{state-trans-L}
\end{equation}

Now, using the block matrix $\left( {\mathcal{L}},{\mathcal{B}}_{2}\right)$,
we transform the state in the control problem (\ref{or-eq})-(\ref{perf-ind})
as follows:
\begin{equation}
Z(t)=\left( {\mathcal{L}} , {\mathcal{B}}_{2}\right) z(t),
\label{state-trans0}
\end{equation}%
where $z(t)\in E^{n}$ is a new state.

By virtue of the results of \cite{GlizerFridmanTuretsky2007}, the
transformation (\ref{state-trans0}) is nonsingular.

\begin{remark}
\label{zero-matr-notation}In what follows, we use the notation $%
O_{n_{1}\times n_{2}}$ for the zero matrix of dimension $n_{1}\times n_{2}$,
excepting the cases where the dimension of zero matrix is obvious. In such
cases, we use the notation $0$ for the zero matrix. By $I_{m}$, we denote
the identity matrix of dimension $m$.
\end{remark}

Based on the results of \cite{Glizer-Kelis2015} (Lemma 1), we have the
following lemma.

\begin{lemma}
\label{init-probl-transf} Let the assumptions (A1), (A2), (A4), (A5) be
valid. Then, transforming the state variable of the problem (\ref{or-eq})-(%
\ref{perf-ind}) in accordance with (\ref{state-trans0}), and redenoting the
control as $u(t)$, we obtain the control problem with the dynamics
\begin{equation}
\frac{dz(t)}{dt}=Az(t)+Bu(t)+f(t),\ \ \ \ z(0)=z_{0},\ \ \ t\geq 0,
\label{new-eq}
\end{equation}%
and the cost functional
\begin{equation}
J(u)=\int\limits_{0}^{+ \infty }\left[ z^{T}(t)Dz(t)+u^{T}(t)Gu(t)\right] dt,
\label{new-perf-ind}
\end{equation}%
where
\begin{equation}
A=\left( {\mathcal{L}},{\mathcal{B}}_{2}\right)^{-1}{\mathcal{A}}\left( {%
\mathcal{L}},{\mathcal{B}}_{2}\right) ,  \label{new-matr-A}
\end{equation}%
\begin{equation}
B=\left( {\mathcal{L}},{\mathcal{B}}_{2}\right) ^{-1}{\mathcal{B}}=\left(
\begin{array}{l}
B_{1} \\
B_{2}%
\end{array}%
\right) ,  \label{new-matr-B}
\end{equation}%
\begin{equation}
B_{1}=\left(
\begin{array}{c}
O_{\left( n-r\right) \times r} \\
\widetilde{I}_{1}%
\end{array}%
\right) ,\ \ \ \ \widetilde{I}_{1}=\left(
\begin{array}{cc}
I_{q}\ ,\ O_{q\times \left( r-q\right) } &
\end{array}%
\right) ,  \label{newmatrB1}
\end{equation}%
\begin{equation}
B_{2}={\mathcal{H}}B_{1}+\widetilde{I}_{2},\ \ \ \ \widetilde{I}_{2}=\left(
\begin{array}{cc}
O_{\left( r-q\right) \times q}\ ,\ I_{r-q} &
\end{array}%
\right) ,  \label{newmatrB2}
\end{equation}%
\begin{equation}
\begin{array}{ll}
D=\left( {\mathcal{L}},{\mathcal{B}}_{2}\right) ^{T}{\mathcal{D}}\left({%
\mathcal{L}},{\mathcal{B}}_{2}\right) =\left(
\begin{array}{l}
D_{1}\ \ \ \ \ \ \ \ \ \ \ \ \ \ \ \ \ \ \ O_{(n-r+q)\times (r-q)} \\
O_{(r-q)\times (n-r+q)}\ \ \ D_{2}%
\end{array}%
\right) , &  \\
&  \\
D_{1}\ ={\mathcal{L}}^{T}{\mathcal{DL}},\ \ \ \ D_{2}={\mathcal{B}}_{2}^{T}{%
\mathcal{DB}}_{2}, &
\end{array}
\label{new-matr-D}
\end{equation}%
\begin{equation}  \label{f(t)}
f(t)=\left( {\mathcal{L}},{\mathcal{B}}_{2}\right)^{-1}\mathcal{F}(t),
\end{equation}
\begin{equation}
z_{0}=\left( {\mathcal{L}},{\mathcal{B}}_{2}\right) ^{-1}Z_{0},
\label{new-z0}
\end{equation}
the matrices $D_{1}$ and $D_{2}$ are symmetric, and $D_{1}$ is positive
semi-definite ($D_{1} \ge 0$), $D_{2}$ is positive definite ($D_{2} > 0$).
Moreover, the function $f(t)$ satisfies the inequality
\begin{equation}  \label{f-ineq}
\|f(t)\|\le a\exp(-\gamma t), \ \ \ \ t \ge 0,
\end{equation}
where $a>0$ is some constant.
\end{lemma}

\begin{remark}
\label{new-problem} In the optimal control problem (\ref{new-eq})-(\ref%
{new-perf-ind}), the cost functional $J(u)$ is minimized by the control $%
u(t) $. Since the weight matrix of the control cost in the cost functional $%
J(u)$ is singular, the solution (if any) of this game can be obtained
neither by the Pontriagin's Maximum Principle nor by the
Hamilton-Jacobi-Bellman equation method, meaning that the problem (\ref%
{new-eq})-(\ref{new-perf-ind}) is singular. The set $M_{u}$ of admissible
state-feedback controls $u(z,t)$ in the problem (\ref{new-eq})-(\ref%
{new-perf-ind}) is defined similarly to such a set $\mathcal{M}_{U}$ in the
problem (\ref{or-eq})-(\ref{perf-ind}). The infimum
\begin{equation}  \label{inf-J}
J^{*} \overset{\triangle}{=} \inf_{u(z,t)\in M_{u}}J\big(u(z,t)\big)
\end{equation}
is nonnegative. Moreover, if ${M}_{u}\neq\emptyset$, this infimum is finite.
The minimizing control sequence $\{u_{k}(z,t)\}$ and the optimal
state-feedback control $u^{*}(z,t)$ in this problem are defined similarly to
those in the problem (\ref{or-eq})-(\ref{perf-ind}), (see (\ref{minim-seq})
and (\ref{opt-minim-strat-def}), respectively).
\end{remark}

\subsection{Equivalence of the problems (\protect\ref{or-eq})-(\protect\ref%
{perf-ind}) and (\protect\ref{new-eq})-(\protect\ref{new-perf-ind}) \label%
{Sec2.4}}

\begin{lemma}
\label{nonempty}Let the assumptions (A1), (A2), (A5) be valid. Then, the
sets $M_{u}$ and ${\mathcal{M}}_{U}$ of the admissible state-feedback
controls in the problems (\ref{new-eq})-(\ref{new-perf-ind}) and (\ref{or-eq}%
)-(\ref{perf-ind}), respectively, are either both nonempty or both empty.
\end{lemma}

\begin{proof}
The statement of the lemma directly follows from Lemma \ref%
{init-probl-transf}, the invertibility of the transformation (\ref%
{state-trans0}), and the definitions of the sets $\mathcal{M}_{U}$ and $%
M_{u} $.
\end{proof}

\begin{lemma}
\label{equal-inf}Let the assumptions (A1), (A2), (A5) be valid. Let $%
M_{u}\neq\emptyset$. Then, the infimum values $J^{*}$ and $\mathcal{J}^{*}$
of the cost functionals in the problems (\ref{new-eq})-(\ref{new-perf-ind})
and (\ref{or-eq})-(\ref{perf-ind}), respectively, are finite and equal to
each other.
\end{lemma}

\begin{proof}
The finiteness of $J^{*}$ and $\mathcal{J}^{*}$ follows immediately from the
condition $M_{u}\neq\emptyset$, Lemma \ref{nonempty}, and Remarks \ref%
{finiteness-mathcal-J^*} and \ref{new-problem}.

Proceed to the proof of the equality
\begin{equation}  \label{equal}
J^{*}=\mathcal{J}^{*}.
\end{equation}
Due to the definition of $J^{*}$, there exists a control sequence $%
\{u_{k}(z,t)\}$, $u_{k}(z,t)\in M_{u}$, $(k=1,2,...)$, such that
\begin{equation}  \label{lim-1-lem-2.3}
\lim_{k\rightarrow +\infty}J\big(u_{k}(z,t)\big)=J^{*}.
\end{equation}
Similarly, due to the definition of $\mathcal{J}^{*}$, there exists a
control sequence $\{\widehat{U}_{k}(Z,t)\}$, $\widehat{U}_{k}(Z,t)\in
\mathcal{M}_{U}$, $(k=1,2,...)$, such that
\begin{equation}  \label{lim-2-lem-2.3}
\lim_{k\rightarrow +\infty}\mathcal{J}\big(\widehat{U}_{k}(Z,t)\big)=%
\mathcal{J}^{*}.
\end{equation}

Let us define the controls
\begin{equation}
U_{k}(Z,t)\overset{\triangle }{=}u_{k}\big((\mathcal{L},\mathcal{B}%
_{2})^{-1}Z,t\big),\ \ \ \ \hat{u}_{k}(z,t)\overset{\triangle }{=}\widehat{U}%
_{k}\big((\mathcal{L},\mathcal{B}_{2})z,t\big),\ \ \ k=1,2,...\ .
\label{new-strategies}
\end{equation}%
Using Lemma \ref{init-probl-transf}, the invertibility of the transformation
(\ref{state-trans0}) and the definitions of the sets $\mathcal{M}_{U}$ and $%
M_{u}$, one directly obtains that
\begin{equation}
U_{k}(Z,t)\in \mathcal{M}_{U},\ \ \ \ \hat{u}_{k}(z,t)\in M_{u},\ \ \
k=1,2,...\ ,  \label{belong-1}
\end{equation}%
and
\begin{equation}
\mathcal{J}\big(U_{k}(Z,t)\big)=J\big(u_{k}(z,t)\big),\ \ \ \ J\big(\hat{u}%
_{k}(z,t)\big)=\mathcal{J}\big(\widehat{U}_{k}(Z,t)\big),\ \ \ k=1,2,...\ .
\label{equals-1}
\end{equation}%
The equations (\ref{lim-1-lem-2.3})-(\ref{lim-2-lem-2.3}) and (\ref{equals-1}%
), as well as the facts that ${\mathcal{J}}^{\ast }$ is the infimum of the
cost functional in the problem (\ref{or-eq})-(\ref{perf-ind}) and $J^{\ast }$
is the infimum of the cost functional in the problem (\ref{new-eq})-(\ref%
{new-perf-ind}), imply the inequalities
\begin{equation}
\mathcal{J}^{\ast }\leq J^{\ast },\ \ \ \ J^{\ast }\leq \mathcal{J}^{\ast },
\label{inequals-1}
\end{equation}%
which yield the equality (\ref{equal}). Thus, the lemma is proven.
\end{proof}

In the sequel of this paper, we deal with the optimal control problem (\ref%
{new-eq})-(\ref{new-perf-ind}). We call this problem the Original Optimal
Control Problem (OOCP). As it was mentioned above, the OOCP is singular.
Moreover, this problem does not have, in general, an optimal control among
regular functions.

\subsection{Objectives of the paper \label{Sec2.5}}

The objectives of this paper are: \newline
(I) to establish finiteness of the infimum of the cost functional in the
OOCP;\newline
(II) to derive an expression for this infimum;\newline
(III) to design a minimizing sequence of state-feedback controls in the
OOCP. 

\section{Regularization of the OOCP}

\subsection{Partial cheap control problem}

Consider the optimal control problem with the dynamics (\ref{new-eq}) and
the performance index
\begin{equation}
J_{\varepsilon}(u)\overset{\triangle}{=}\int\limits_{0}^{+ \infty }\left[
z^{T}(t)Dz(t)+u^{T}(t)(G+\mathcal{E})u(t)\right]dt\rightarrow\min_{u},
\label{perf-ind-auxil-pr-regul}
\end{equation}
where
\begin{equation}
{\mathcal{E}}=\mathrm{diag}\Big(\underbrace{0,...,0}_{q},\underbrace{%
\varepsilon ^{2},...,\varepsilon ^{2}}_{r-q}\Big),  \label{mathcal-E}
\end{equation}
and $\varepsilon >0$ is a small parameter.

\begin{remark}
Since the parameter $\varepsilon >0$ is small, the problem (\ref{new-eq}), (%
\ref{perf-ind-auxil-pr-regul}) is a partial cheap control problem, i.e., an
optimal control problem where a cost of some control coordinates in the cost
functional is much smaller than costs of the state and the other control
coordinates. In what follows, we call this problem the Partial Cheap Control
Problem (PCCP). The case of the completely cheap control was widely studied
in the literature for various control problems (see e.g. \cite%
{Glizer2014,omal-jam-77,sab-san-87,ser-br-kok-may-99,Glizer1999,sme-sob-05,Glizer2009}
and references therein). However, to the best of the authors' knowledge, the
partial cheap control case was analyzed only in two works. Namely, in \cite%
{OReil83} a two-time-scale decomposition of the time-invariant regulator
problem with partial cheap control was carried out, yielding a near-optimal
composite control. In \cite{Glizer-Kelis2015}, a finite-horizon zero-sum
linear-quadratic differential game with partial cheap control of the
minimizing player was studied.
\end{remark}

\subsection{Optimal state-feedback control of the PCCP}

We look for such a control in the same set of the admissible controls as was
introduced earlier for the OOCP, i.e., in the set $M_{u}$.

Consider the algebraic matrix Riccati equation
\begin{equation}
PA+A^{T}P-PS(\varepsilon)P+D=0,\   \label{eq-P}
\end{equation}
where
\begin{equation}
S(\varepsilon )=B(G+\mathcal{E})^{-1}B^{T}.  \label{S_u}
\end{equation}

By virtue of the inequality (\ref{f-ineq}) and the results of \cite%
{Salukvadze-62}, if for a given $\varepsilon >0$ the equation (\ref{eq-P})
has a symmetric solution $P^{*}(\varepsilon)\ge 0$ such that the matrix
\begin{equation}  \label{mathcal-A}
\mathcal{A}(\varepsilon)\overset{\triangle}{=} A-S(\varepsilon)P^{*}(%
\varepsilon)
\end{equation}
is a Hurwitz one, then the optimal control of the PCCP exists in the set of
the admissible state-feedback controls $M_{u}$. This control is unique and
it has the form
\begin{equation}  \label{opt-contr-auxil}
u^{*}_{\varepsilon}(z,t) = -(G+\mathcal{E})^{-1}B^{T}P^{*}(\varepsilon)z -
(G+\mathcal{E})^{-1}B^{T}h(t),
\end{equation}
where the $n$-dimensional vector-valued function $h(t)$, $t\in[0,+\infty)$
is the unique solution of the terminal-value problem
\begin{equation}  \label{dif-eq-h}
\frac{dh(t)}{dt}=-\mathcal{A}^{T}(\varepsilon)h(t)-P^{*}(\varepsilon)f(t),\
\ \ \ \ h(+\infty)=0.
\end{equation}

The optimal value of the cost functional in the PCCP has the form
\begin{equation}  \label{opt-tilde-J}
J_{\varepsilon}^{*}= z_{0}^{T}P^{*}(\varepsilon)z_{0} +2h^{T}(0)z_{0}+s(0),
\end{equation}
where the scalar function $s(t)$, $t\in[0,+\infty)$ is the unique solution
of the terminal-value problem
\begin{equation}  \label{eq-s}
\frac{ds(t)}{dt}=-2h^{T}(t)f(t)+h^{T}(t)S(\varepsilon)h(t),\ \ \ \
s(+\infty)=0.
\end{equation}

\section{Asymptotic Analysis of the PCCP}

\subsection{Asymptotic solution of the equation (\protect\ref{eq-P})}

First of all, let as note that the matrix $S(\varepsilon )$, appearing in
this equation, can be represented (similarly to the results of \cite%
{Glizer-Kelis2015}) in the following block form:
\begin{equation}
S(\varepsilon )=\left(
\begin{array}{lr}
S_{1}\ \  & S_{2} \\
&  \\
S_{2}^{T}\ \  & (1/\varepsilon ^{2})S_{3}(\varepsilon )%
\end{array}%
\right),  \label{S_u-new-form}
\end{equation}
where
\begin{eqnarray}
S_{1}=\left(
\begin{array}{l}
O_{(n-r)\times (n-r)}\ \ O_{(n-r)\times q} \\
O_{q\times (n-r)}\ \ \ \ \ \ \ \widetilde{G}^{-1}%
\end{array}%
\right),\ \ \ S_{2}=\left(%
\begin{array}{c}
O_{(n-r)\times (r-q)} \\
H_{1}%
\end{array}%
\right),  \notag \\
\   \notag \\
S_{3}(\varepsilon)=\varepsilon ^{2}H_{2}+I_{r-q},  \label{eqL-1}
\end{eqnarray}
\begin{eqnarray}
\widetilde{G}=\mathrm{diag}\big(g_{1},...,g_{q}\big),\ \ H_{1}=H_{3}{%
\mathcal{H}}^{T},\ \ H_{2}={\mathcal{H}}\left(%
\begin{array}{c}
O_{(n-r)\times (r-q)} \\
H_{1}%
\end{array}%
\right),  \notag \\
\   \notag \\
H_{3}=\Big(O_{q\times (n-r)},\widetilde{G}^{-1}\Big),  \label{eqH1H2}
\end{eqnarray}
${\mathcal{H}}$ is defined in (\ref{state-trans-L}).

Due to (\ref{S_u-new-form})-(\ref{eqH1H2}), the left-hand side of the
equation (\ref{eq-P}) has a singularity at $\varepsilon =0$. To remove this
singularity, we seek the symmetric solution $P(\varepsilon )$ of the
equation (\ref{eq-P}) in the block form
\begin{equation}
P(\varepsilon )=\left(%
\begin{array}{cc}
P_{1}(\varepsilon )\  & \varepsilon P_{2}(\varepsilon ) \\
&  \\
\varepsilon P_{2}^{T}(\varepsilon ) & \ \varepsilon P_{3}(\varepsilon)%
\end{array}%
\right) ,  \label{tilde-P-block-form}
\end{equation}
where the blocks $P_{1}(\varepsilon )$, $P_{2}(\varepsilon )$ and $%
P_{3}(\varepsilon )$ have the dimensions $(n-r+q)\times (n-r+q)$, $%
(n-r+q)\times \left( r-q\right) $ and $\left(r-q\right) \times \left(
r-q\right) $, respectively; and
\begin{equation}  \label{symmetry-tilde-P_1-P_3}
P_{1}^{T}(\varepsilon )=P_{1}(\varepsilon ),\ \ \ P_{3}^{T}(\varepsilon
)=P_{3}(\varepsilon ).
\end{equation}

We also partition the matrix $A$ into blocks as follows:
\begin{equation}  \label{A-block-form}
A=\left(
\begin{array}{cc}
A_{1}\  & A_{2} \\
&  \\
A_{3}\ \  & A_{4}%
\end{array}%
\right),
\end{equation}
where the blocks $A_{1}$, $A_{2}$, $A_{3}$ and $A_{4}$ have the dimensions $%
(n-r+q)\times (n-r+q)$, $(n-r+q)\times \left( r-q\right)$, $(r-q)\times
(n-r+q)$ and $\left(r-q\right) \times \left( r-q\right)$, respectively.

Substitution of the block representations for the matrices $D$, $%
S(\varepsilon)$, $P(\varepsilon)$, and $A$ (see (\ref{new-matr-D}), (\ref%
{S_u-new-form}), (\ref{tilde-P-block-form}), and (\ref{A-block-form})) into
the equation (\ref{eq-P}) yields after a routine rearrangement the following
equivalent set of Riccati-type algebraic matrix equations with respect to $%
P_{1}(\varepsilon)$, $P_{2}(\varepsilon)$ and $P_{3}(\varepsilon)$:
\begin{eqnarray}
P_{1}(\varepsilon )A_{1}+\varepsilon P_{2}(\varepsilon
)A_{3}+A_{1}^{T}P_{1}(\varepsilon )+\varepsilon
A_{3}^{T}P_{2}^{T}(\varepsilon ) -P_{1}(\varepsilon )S_{1}P_{1}(\varepsilon)
\notag \\
-\varepsilon P_{2}(\varepsilon )S_{2}^{T}P_{1}(\varepsilon )-\varepsilon
P_{1}(\varepsilon )S_{2}P_{2}^{T}(\varepsilon ) - P_{2}(\varepsilon
)S_{3}(\varepsilon)P_{2}^{T}(\varepsilon )+D_{1}=0,  \label{eq-P_1}
\end{eqnarray}
\begin{eqnarray}
P_{1}(\varepsilon )A_{2}+\varepsilon P_{2}(\varepsilon )A_{4}+\varepsilon
A_{1}^{T}P_{2}(\varepsilon )+\varepsilon A_{3}^{T}P_{3}(\varepsilon )
-\varepsilon P_{1}(\varepsilon )S_{1}P_{2}(\varepsilon )  \notag \\
- \varepsilon^{2}P_{2}(\varepsilon ) S_{2}^{T} P_{2}(\varepsilon
)-\varepsilon P_{1}(\varepsilon )S_{2}P_{3}(\varepsilon ) -P_{2}(\varepsilon
)S_{3}(\varepsilon)P_{3}(\varepsilon )=0,  \label{eq-P_2}
\end{eqnarray}
\begin{eqnarray}
\varepsilon P_{2}^{T}(\varepsilon )A_{2}+\varepsilon P_{3}(\varepsilon
)A_{4}+\varepsilon A_{2}^{T}P_{2}(\varepsilon )+\varepsilon
A_{4}^{T}P_{3}(\varepsilon ) -\varepsilon ^{2}P_{2}^{T}(\varepsilon
)S_{1}P_{2}(\varepsilon )  \notag \\
-\varepsilon ^{2}P_{3}(\varepsilon )S_{2}^{T}P_{2}(\varepsilon )-\varepsilon
^{2}P_{2}^{T}(\varepsilon )S_{2}P_{3}(\varepsilon ) - P_{3}(\varepsilon
)S_{3}(\varepsilon)P_{3}(\varepsilon )+D_{2}=0.  \label{eq-P_3}
\end{eqnarray}

We seek the asymptotic solution $P_{i,m}^{\mathrm{as}}(\varepsilon )$, $%
(i=1,2,3)$ of the system (\ref{eq-P_1})-(\ref{eq-P_3}) in the form $P_{i,m}^{%
\mathrm{as}}(\varepsilon )=P_{i0}+\varepsilon P_{i1}+...+\varepsilon
^{m}P_{i,m}$, where $m\geq 0$ is a given integer. In what follows, we
restrict ourselves to the case of zero-order asymptotic solution, i.e., $m=0$
and
\begin{equation}
P_{i,0}^{\mathrm{as}}(\varepsilon )=P_{i0},\ \ \ i=1,2,3.
\label{P-asympt-sol}
\end{equation}%
Equations for the zero-order asymptotic solution terms are obtained by
substitution of (\ref{P-asympt-sol}) into (\ref{eq-P_1})-(\ref{eq-P_3})
instead of $P_{i}(\varepsilon )$, $(i=1,2,3)$, and equating coefficients for
the zero power of $\varepsilon$ on both sides of the resulting equations.
Thus, we have the following system:
\begin{equation}
P_{10}A_{1}+A_{1}^{T}P_{10}-P_{10}S_{1}P_{10} - P_{20}P_{20}^{T}+D_{1}=0,
\label{eq-tilde-P_10^o}
\end{equation}
\begin{equation}
P_{10}A_{2}-P_{20}P_{30}=0,  \label{eq-tilde-P_20^o}
\end{equation}
\begin{equation}
\big( P_{30}\big)^{2}-D_{2}=0.  \label{eq-tilde-P_30^o}
\end{equation}
Solving the equations (\ref{eq-tilde-P_30^o}) and (\ref{eq-tilde-P_20^o})
with respect to $P_{3 0}$ and $P_{2 0}$, we obtain
\begin{equation}  \label{tilde-P_320}
P_{3 0}=P_{3 0}^{*}\overset{\triangle}{=}\big(D_{2}\big)^{1/2},\ \ \ \ P_{2
0}=P_{1 0}A_{2}\big(D_{2}\big)^{-1/2},
\end{equation}
where the superscript $"1/2"$ denotes the unique symmetric positive definite
square root of the corresponding symmetric positive definite matrix, while
the superscript $"-1/2"$ denotes the inverse matrix for such a square root.

Since $\big(D_{2}\big)^{1/2}$ is positive definite, then there exists a
positive number $\beta$ such that all eigenvalues $\lambda\Big(-\big(D_{2}%
\big)^{1/2}\Big)$ of the matrix $-\big(D_{2}\big)^{1/2}$ satisfy the
inequality
\begin{equation}  \label{eigenval-ineq-1}
\mathrm{Re}\lambda\Big(-\big(D_{2}\big)^{1/2}\Big)< -\beta.
\end{equation}

Substitution of the expression for $P_{2 0}$ from (\ref{tilde-P_320}) into (%
\ref{eq-tilde-P_10^o}) yields the algebraic matrix Riccati equation with
respect to $P_{1 0}$
\begin{equation}
P_{10}A_{1}+A_{1}^{T}P_{10}-P_{10}{S}_{0}P_{10}+D_{1}=0,
\label{eq-P_10^o-new-form}
\end{equation}%
where
\begin{equation}
{S}_{0}=A_{2}D_{2}^{-1}A_{2}^{T}+S_{1}.  \label{S_1^o}
\end{equation}
Based on the results of \cite{Glizer-Kelis2015}, we can represent the matrix
$S_{0}$ in the form
\begin{equation}
{S}_{0}=\bar{B}\Theta ^{-1}\bar{B}^{T},  \label{S_1^o-new-repr}
\end{equation}
where
\begin{equation}
\bar{B}=\left(
\begin{array}{l}
\widetilde{B}\ ,\ A_{2}%
\end{array}%
\right),\ \ \ \ \widetilde{B}=\left(
\begin{array}{c}
O_{\left( n-r\right) \times q} \\
I_{q}%
\end{array}%
\right),  \label{B^o}
\end{equation}
\begin{equation}
\Theta =\left(
\begin{array}{cc}
\widetilde{G} & O_{q\times \left( r-q\right) } \\
O_{\left( r-q\right) \times q} & D_{2}%
\end{array}%
\right).  \label{Theta}
\end{equation}

Let $F_{1}$ be a matrix such that
\begin{equation}
D_{1}=F_{1}^{T}F_{1}.  \label{F_1^TF_1}
\end{equation}

In what follows, we assume: \newline
\textbf{(A6)} The pair $(A_{1},\bar{B})$ is stabilizable; \newline
\textbf{(A7)} the pair $(A_{1},F_{1})$ is detectable.

Using the equation (\ref{S_1^o-new-repr}), the assumptions (A6), (A7), and
the results of \cite{Anders-Moore1971}, one directly obtains that the
algebraic matrix Riccati equation (\ref{eq-P_10^o-new-form}) has the unique
symmetric solution $P_{1 0}^{*}\ge 0$. Moreover, the matrix
\begin{equation}  \label{matr-tildA_10}
\mathcal{A}_{0}\overset{\triangle}{=}A_{1}-S_{0}P_{1 0}^{*}
\end{equation}
is a Hurwitz one. Therefore, there exists a positive number $\alpha$, ($%
\alpha\ne\gamma$), such that all eigenvalues $\lambda\Big(\mathcal{A}_{1 0}%
\Big)$ of this matrix satisfy the inequality
\begin{equation}  \label{eigenval-ineq-2}
\mathrm{Re}\lambda\Big(\mathcal{A}_{0}\Big)<-\alpha.
\end{equation}

Now, using Lemma \ref{init-probl-transf} (the positive definiteness of the
matrix $D_{2}$), the above mentioned features of the equation (\ref%
{eq-P_10^o-new-form}), and the results of \cite{Kokotovic} (Sections 3.4 and
3.6.1), we can state the following:

\begin{lemma}
\label{justif-asymp-sol-Riccati}Let the assumptions (A1)-(A3), (A5)-(A7) be
valid. Then, there exists a positive number $\varepsilon_{0}$ such that for
all $\varepsilon\in[0,{\varepsilon}_{0}]$ the equation (\ref{eq-P}) has the
unique symmetric solution $P^{*}(\varepsilon)\ge 0$. This solution has the
block form
\begin{equation}
P^{\ast }(\varepsilon )=\left(
\begin{array}{cc}
P_{1}^{\ast }(\varepsilon ) & \varepsilon{\ P}_{2}^{\ast }(\varepsilon ) \\
\varepsilon \big(P_{2}^{\ast }(\varepsilon )\big)^{T} & \varepsilon
P_{3}^{\ast }(\varepsilon )%
\end{array}%
\right) ,  \label{P^*(eps)-block}
\end{equation}%
where the blocks $P_{1}^{\ast }(\varepsilon )$, $P_{2}^{\ast }(\varepsilon )$%
, $P_{3}^{\ast }(\varepsilon )$ are of dimensions $(n-r+q)\times (n-r+q)$, $%
(n-r+q)\times(r-q)$, $(r-q)\times(r-q)$, respectively. These blocks satisfy
the inequalities
\begin{equation}
\Vert P_{i}^{\ast }(\varepsilon )-P_{i0}^{\ast }\Vert \leq a\varepsilon ,\ \
\ i=1,2,3,\ \ \ \varepsilon \in [0,{\varepsilon}_{0}],
\label{ineq-tilde(P-P_0)}
\end{equation}%
where $P_{2 0}^{*}=P_{1 0}^{*}A_{2}\big(D_{2}\big)^{-1/2}$; $a>0$ is some
constant independent of $\varepsilon$; $\|\cdot\|$ denotes the Euclidean
norm either of a matrix, or of a vector.

Moreover, the matrix $\mathcal{A}(\varepsilon)$, given by (\ref{mathcal-A}),
is a Hurwitz one.
\end{lemma}

\subsection{Asymptotic solution of the problem (\protect\ref{dif-eq-h})}

Using the equations (\ref{mathcal-A}) and (\ref{S_u-new-form}), (\ref%
{A-block-form}), (\ref{P^*(eps)-block}), we can represent the matrix $%
\mathcal{A}(\varepsilon)$ in the block form
\begin{equation}  \label{mathcal-A-blocks}
\mathcal{A}(\varepsilon)= \left(%
\begin{array}{l}
{\mathcal{A}}_{1}(\varepsilon)\ \ \ \ \ \ \ \ \ \ \ \ \ {\mathcal{A}}%
_{2}(\varepsilon) \\
(1/\varepsilon){\mathcal{A}}_{3}(\varepsilon)\ \ \ \ (1/\varepsilon){%
\mathcal{A}}_{4}(\varepsilon)%
\end{array}%
\right),
\end{equation}
where
\begin{equation}  \label{mathcal-A_1}
{\mathcal{A}}_{1}(\varepsilon)=A_{1}- S_{1}P^{*}_{1}(\varepsilon)-
\varepsilon S_{2}\big(P^{*}_{2}(\varepsilon)\big)^{T},
\end{equation}
\begin{equation}  \label{mathcal-A_2}
{\mathcal{A}}_{2}(\varepsilon)=A_{2}-\varepsilon
S_{1}P^{*}_{2}(\varepsilon)- \varepsilon S_{2}P^{*}_{3}(\varepsilon),
\end{equation}
\begin{equation}  \label{mathcal-A_3}
{\mathcal{A}}_{3}(\varepsilon)=\varepsilon A_{3}- \varepsilon
S_{2}P^{*}_{1}(\varepsilon)-S_{3}(\varepsilon)\big(P^{*}_{2}(\varepsilon)%
\big)^{T},
\end{equation}
\begin{equation}  \label{mathcal-A_4}
{\mathcal{A}}_{4}(\varepsilon)=\varepsilon
A_{4}-\varepsilon^{2}S_{2}^{T}P^{*}_{2}(\varepsilon)-
S_{3}(\varepsilon)P^{*}_{3}(\varepsilon).
\end{equation}

Also, let us partition the vector-valued function $f(t)$ into blocks as
follows:
\begin{equation}  \label{f-block}
f(t)=\left(%
\begin{array}{c}
f_{1}(t) \\
f_{2}(t)%
\end{array}%
\right),\ \ \ \ t\ge 0,
\end{equation}
where the blocks $f_{1}(t)$ and $f_{2}(t)$ are of the dimensions $n-r+q$ and
$r-q$, respectively.

We look for the solution of the problem (\ref{dif-eq-h}) in the block form
\begin{equation}  \label{h-block}
h(t,\varepsilon)=\left(%
\begin{array}{l}
h_{1}(t,\varepsilon) \\
\varepsilon h_{2}(t,\varepsilon)%
\end{array}%
\right),\ \ \ \ t \ge 0,
\end{equation}
where the blocks $h_{1}(t,\varepsilon)$ and $h_{2}(t,\varepsilon)$ are of
the dimensions $n-r+q$ and $r-q$, respectively.

Substitution of the block representations for $\mathcal{A}(\varepsilon)$, $%
P^{*}(\varepsilon)$, $f(t)$ and $h(t,\varepsilon)$ into the problem (\ref%
{dif-eq-h}) yields the following initial-value problem, equivalent to (\ref%
{dif-eq-h}):
\begin{equation}  \label{eq-h_1}
\frac{dh_{1}(t,\varepsilon)}{dt}=-{\mathcal{A}}_{1}^{T}(\varepsilon)h_{1}(t,%
\varepsilon)-{\mathcal{A}}_{3}^{T}(\varepsilon)h_{2}(t,\varepsilon)-
P_{1}^{*}(\varepsilon)f_{1}(t)-\varepsilon P_{2}^{*}(\varepsilon)f_{2}(t),
\end{equation}
\begin{equation}  \label{eq-h_2}
\varepsilon\frac{dh_{2}(t,\varepsilon)}{dt}=-{\mathcal{A}}%
_{2}^{T}(\varepsilon)h_{1}(t,\varepsilon)-{\mathcal{A}}_{4}^{T}(%
\varepsilon)h_{2}(t,\varepsilon)- \varepsilon \big(P_{2}^{*}(\varepsilon)%
\big)^{T}f_{1}(t)-\varepsilon P_{3}^{*}(\varepsilon)f_{2}(t),
\end{equation}
\begin{equation}  \label{termin-cond}
h_{1}(+\infty,\varepsilon)=0,\ \ \ \ \ h_{2}(+\infty,\varepsilon)=0.
\end{equation}
Let us construct the zero-order asymptotic solutions $\{h_{1 0}(t) , h_{2
0}(t)\}$ of the problem (\ref{eq-h_1})-(\ref{termin-cond}). The equations
for this asymptotic solution are obtained from the system (\ref{eq-h_1})-(%
\ref{eq-h_2}) by setting there formally $\varepsilon=0$ and using Lemma \ref%
{justif-asymp-sol-Riccati}. Thus, we have
\begin{equation}  \label{eq-h_10}
\frac{dh_{1 0}(t)}{dt}=-{\mathcal{A}}_{1}^{T}(0)h_{1 0}(t)-{\mathcal{A}}%
_{3}^{T}(0)h_{2 0}(t)- P_{1 0}^{*}f_{1}(t),
\end{equation}
\begin{equation}  \label{eq-h_20}
0=-{\mathcal{A}}_{2}^{T}(0)h_{1 0}(t)-{\mathcal{A}}_{4}^{T}(0)h_{2 0}(t).
\end{equation}
The terminal condition for $h_{1 0}(t)$ is obtained from the condition for $%
h_{1}(t,\varepsilon)$ (see (\ref{termin-cond})) by formal replacing there $%
h_{1}(+\infty,\varepsilon)$ with $h_{1 0}(+\infty)$, i.e.,
\begin{equation}  \label{term-cond-h_10}
h_{1 0}(+\infty)=0.
\end{equation}

Solving the equation (\ref{eq-h_20}) with respect to $h_{2 0}(t)$, and
taking into account that ${\mathcal{A}}_{4}(0)=-P_{3 0}^{*}=-\big(D_{2}\big)%
^{1/2}$ and ${\mathcal{A}}_{2}(0)=A_{2}$, we obtain
\begin{equation}  \label{h_20}
h_{2 0}(t)=\big(D_{2}\big)^{-1/2}A_{2}^{T}h_{1 0}(t).
\end{equation}
Substitution of (\ref{h_20}) into (\ref{eq-h_10}) and using the expressions $%
{\mathcal{A}}_{1}(0)=A_{1}-S_{1}P_{1 0}^{*}$, ${\mathcal{A}}_{3}(0)=-\big(%
P_{2 0}^{*}\big)^{T}$, as well as the equations (\ref{S_1^o}), (\ref%
{matr-tildA_10}) and the expression for $P_{2 0}^{*}$ (see Lemma \ref%
{justif-asymp-sol-Riccati}), yield the differential equation for $h_{1 0}(t)$
\begin{equation}  \label{dif-eq-h_10}
\frac{dh_{1 0}(t)}{dt}=-{\mathcal{A}}_{0}^{T}h_{1 0}(t)-P_{1 0}^{*}f_{1}(t).
\end{equation}
Due to the inequalities (\ref{f-ineq}) and (\ref{eigenval-ineq-2}), this
equation, subject to the condition (\ref{term-cond-h_10}), has the unique
solution
\begin{equation}  \label{h_10}
h_{1 0}(t)=\int_{0}^{+\infty}\exp\big({\mathcal{A}}_{0}^{T}\zeta\big)P_{1
0}^{*}f_{1}(\zeta+t)d\zeta,\ \ \ \ t \ge 0
\end{equation}
satisfying the inequality
\begin{equation}  \label{ineq-h_10}
\big\|h_{1 0}(t)\big\|\le a\exp(-\gamma t),\ \ \ \ t \ge 0,
\end{equation}
where $a>0$ is some constant.

The equation (\ref{h_20}), along with the condition (\ref{term-cond-h_10})
and the inequality (\ref{ineq-h_10}), yields
\begin{equation}  \label{term-cond-h_20}
h_{2 0}(+\infty)=0,
\end{equation}
\begin{equation}  \label{ineq-h_20}
\big\|h_{2 0}(t)\big\|\le a\exp(-\gamma t),\ \ \ \ t \ge 0,
\end{equation}
where $a>0$ is some constant.

This completes the formal construction of the zero-order asymptotic solution
of the problem (\ref{eq-h_1})-(\ref{termin-cond}).

\begin{lemma}
\label{asympt-h}Let the assumptions (A1)-(A7) be valid. Then, there exists a
positive number $\varepsilon_{1}$, ($\varepsilon_{1}\le\varepsilon_{0}$),
such that for all $\varepsilon\in(0,\varepsilon_{1}]$ the solution $\big\{%
h_{1}(t,\varepsilon) , h_{2}(t,\varepsilon)\big\}$ of the terminal-value
problem (\ref{eq-h_1})-(\ref{termin-cond}) satisfies the inequalities
\begin{equation}  \label{ineq-h_1}
\big\|h_{i}(t,\varepsilon)-h_{i 0}(t)\big\|\le c\varepsilon\exp\left(-\mu
t\right),\ \ \ \ i=1,2,\ \ \ \ t \ge 0,
\end{equation}
where
\begin{equation}  \label{mu}
\mu=\min\{\alpha , \gamma\},
\end{equation}
$c>0$ is some constant independent of $\varepsilon$.
\end{lemma}

The proof of the lemma is presented in Appendix A.

\subsection{Asymptotic solution of the problem (\protect\ref{eq-s})}

Using the block form of the matrix $S(\varepsilon)$, and the vectors $f(t)$
and $h(t,\varepsilon)$ (see (\ref{S_u-new-form}), (\ref{f-block}), (\ref%
{h-block})), we can rewrite equivalently the problem (\ref{eq-s}) as
follows:
\begin{eqnarray}  \label{eq-s-new-form}
\frac{ds(t,\varepsilon)}{dt}=-2\big(h_{1}^{T}(t,\varepsilon)f_{1}(t)+%
\varepsilon h_{2}^{T}(t,\varepsilon)f_{2}(t)\big)+
h_{1}^{T}(t,\varepsilon)S_{1}h_{1}(t,\varepsilon)  \notag \\
+ 2\varepsilon h_{1}^{T}(t,\varepsilon)S_{2}h_{2}(t,\varepsilon) +
h_{2}^{T}(t,\varepsilon)S_{3}(\varepsilon)h_{2}(t,\varepsilon),\ \ \ \
s(+\infty,\varepsilon)=0.
\end{eqnarray}

Let us construct the zero-order asymptotic solutions $s_{0}(t)$ of the
problem (\ref{eq-s-new-form}). The equation for this asymptotic solution is
obtained from the differential equation in (\ref{eq-s-new-form}) by setting
there formally $\varepsilon =0$ and using Lemma \ref{asympt-h}. Thus, we
have
\begin{equation}
\frac{ds_{0}(t)}{dt}%
=-2h_{10}^{T}(t)f_{1}(t)+h_{10}^{T}(t)S_{1}h_{10}(t)+h_{20}^{T}(t)h_{20}(t).
\label{eq-s_0}
\end{equation}%
Substituting the expression for $h_{20}(t)$ (see (\ref{h_20})) into the
right-hand side of (\ref{eq-s_0}), and using (\ref{S_1^o}), we obtain
\begin{equation}
\frac{ds_{0}(t)}{dt}=-2h_{10}^{T}(t)f_{1}(t)+h_{10}^{T}(t)S_{0}h_{10}(t).
\label{eq-s_0-new-form}
\end{equation}%
The terminal condition for $s_{0}(t)$ is obtained from the condition for $%
s(t,\varepsilon )$ in (\ref{eq-s-new-form}) by formal replacing there $%
s(+\infty ,\varepsilon )$ with $s_{0}(+\infty )$, i.e.,
\begin{equation}
s_{0}(+\infty )=0.  \label{term-cond-s_0}
\end{equation}%
The solution of the problem (\ref{eq-s_0-new-form})-(\ref{term-cond-s_0})
has the form
\begin{equation}
s_{0}(t)=\int_{t}^{+\infty }\Big(2h_{10}^{T}(\sigma )f_{1}(\sigma
)-h_{10}^{T}(\sigma )S_{0}h_{10}(\sigma )\Big)d\sigma ,\ \ \ \ t\geq 0.
\label{s_0}
\end{equation}%
Due to the inequalities (\ref{f-ineq}) and (\ref{ineq-h_10}), the integral
in (\ref{s_0}) converges. This completes the formal construction of the
zero-order asymptotic solution of the problem (\ref{eq-s-new-form}).
Similarly to Lemma \ref{asympt-h}, we obtain the following lemma:

\begin{lemma}
\label{asympt-s}Let the assumptions (A1)-(A7) be valid. Then, there exists a
positive number $\varepsilon_{2}$, ($\varepsilon_{2}\le\varepsilon_{1}$),
such that for all $\varepsilon\in(0,\varepsilon_{2}]$ the solution $%
s(t,\varepsilon)$ of the terminal-value problem (\ref{eq-s-new-form})
satisfies the inequality
\begin{equation}  \label{ineq-s}
\big\|s(t,\varepsilon)-s_{0}(t)\big\|\le c\varepsilon,\ \ \ \ \ t \ge 0,
\end{equation}
where $c>0$ is some constant independent of $\varepsilon$.
\end{lemma}

\subsection{Asymptotic expansion of the optimal value of the cost functional}

Let us partition the vector $z_{0}$ into blocks as:
\begin{equation}  \label{z_0-blocks}
z_{0}=\left(%
\begin{array}{c}
x_{0} \\
y_{0}%
\end{array}%
\right),\ \ \ \ x_{0}\in E^{n-r+q},\ \ \ \ y_{0}\in E^{r-q}.
\end{equation}
Let us introduce the value
\begin{equation}  \label{bar-J*}
\bar{J}^{*}\overset{\triangle}{=}x_{0}^{T}P_{1 0}^{*}x_{0}+2h_{1
0}(0)x_{0}+s_{0}(0).
\end{equation}

\begin{lemma}
\label{J_eps*-asympt}Let the assumptions (A1)-(A7) be valid.Then, the
following inequality is satisfied:
\begin{equation}  \label{ineq-J_eps*}
\big|J_{\varepsilon}^{*}-\bar{J}^{*}\big|\le c\varepsilon,\ \ \ \ \
\varepsilon\in(0,\varepsilon_{2}],
\end{equation}
where $J_{\varepsilon}^{*}$ is the optimal value of the cost functional in
the PCCP; $c>0$ is some constant independent of $\varepsilon$.
\end{lemma}

\begin{proof}
The statement of the lemma is a direct consequence of the equation (\ref%
{opt-tilde-J}) and Lemmas \ref{justif-asymp-sol-Riccati}, \ref{asympt-h}, %
\ref{asympt-s}.
\end{proof}

\subsection{Reduced optimal control problem}

Consider the following controlled differential equation:%
\begin{equation}
\frac{d\bar{x}(t)}{dt}=A_{1}\overline{x}(t)+\bar{B}\bar{u}(t)+f_{1}(t),\
t\geq 0,\ \overline{x}(0)=x_{0},  \label{out-dyn1}
\end{equation}%
where $\overline{x}(t)\in E^{n-r+q}$ is the state vector, $\bar{u}(t)\in
E^{r}$ is the control, the matrix $\bar{B}$ is given in (\ref{B^o}).

The cost functional, to be minimized by $\bar{u}(t)$, has the form
\begin{equation}
\bar{J}(\bar{u})= \int_{0}^{+\infty} \big(\bar{x}^{T}(t)D_{1}\bar{x}(t)+\bar{%
u}^{T}(t)\Theta \bar{u}(t)\big)dt,  \label{bar-J}
\end{equation}
where the matrix $\Theta$ is given by (\ref{Theta}).

We call the problem (\ref{out-dyn1})-(\ref{bar-J}) the Reduced Optimal
Control Problem (ROCP).

Consider the set $\bar{{\mathcal{P}}}$ of all functions $\bar{p}(\bar{w}%
,t):E^{n-r+q}\times \left[ 0,+\infty \right) \rightarrow E^{r}$, which are
measurable w.r.t. $t\geq 0$ for any fixed $\bar{w}\in E^{n-r+q}$ and satisfy
the local Lipschitz condition w.r.t. $\bar{w}\in E^{n-r+q}$ uniformly in $%
t\geq 0$.

\begin{definition}
\label{admiss-def-auxil-2}Let $\bar{u}(\bar{x},t)$, $(\bar{x},t)\in
E^{n-r+q}\times \left[ 0,+\infty \right) $, be a function belonging to the
set $\bar{{\mathcal{P}}}$. The function $\bar{u}(\bar{x},t)$ is called an
admissible state-feedback control in the ROCP if the following conditions
hold: (a) the initial-value problem (\ref{out-dyn1}) for $\bar{u}(t)=\bar{u}(%
\bar{x},t)$ has the unique locally absolutely continuous solution $\bar{x}%
(t) $ on the entire interval $\left[ 0,+\infty \right) $; (b) $\bar{x}(t)\in
L^{2}\left[ 0,+\infty ;E^{n-r+q}\right] $; (c) $\bar{u}\big(\bar{x}(t),t\big)%
\in L^{2}\left[ 0,+\infty ;E^{r}\right] $. The set of all such $\bar{u}(\bar{%
x},t)$ is denoted by $\bar{M}_{u}$.
\end{definition}

\begin{lemma}
\label{reduced-pr}Let the assumptions (A1)-(A7) be valid. Then, the optimal
control of the ROCP exists in the set $\bar{M}_{u}$. This control is unique
and it has the form
\begin{equation}  \label{bar-u*}
\bar{u}^{*}(\bar{x},t)=-\Theta^{- 1}\bar{B}^{T}P_{1 0}^{*}\bar{x}-\Theta^{-
1}\bar{B}^{T}h_{1 0}(t).
\end{equation}
The optimal value of the cost functional in the ROCP is $\bar{J}^{*}$,
having the form (\ref{bar-J*}).
\end{lemma}

\begin{proof}
Based on Lemmas \ref{justif-asymp-sol-Riccati}, \ref{asympt-h}, \ref%
{asympt-s}, the statements of the lemma directly follow from the results of
\cite{Salukvadze-62}.
\end{proof}

\begin{remark}
Based on the block form of the matrices $\bar{B}$ and $\Theta$ (see (\ref%
{B^o}) and (\ref{Theta})), the optimal state-feedback control $\bar{u}^{*}(%
\bar{x},t)$ of the ROCP can be represented as follows:
\begin{equation}  \label{u^o*-block-form}
\bar{u}^{\ast}(\bar{x},t)=\left(
\begin{array}{c}
\bar{u}_{1}^{\ast}(\bar{x},t) \\
\bar{u}_{2}^{\ast}(\bar{x},t)%
\end{array}%
\right),
\end{equation}
where
\begin{equation}  \label{u_1^o*}
\bar{u}_{1}^{\ast}(\bar{x},t) = -\widetilde{G}^{-1}\widetilde{B}%
^{T}P_{10}^{*}\bar{x}-\widetilde{G}^{-1}\widetilde{B}^{T}h_{1 0}(t),
\end{equation}
\begin{equation}  \label{u_2^o*}
\bar{u}_{2}^{\ast}(\bar{x},t)= -D_{2}^{-1}A_{2}^{T}P_{10}^{*}\bar{x}%
-D_{2}^{-1}A_{2}^{T}h_{1 0}(t).
\end{equation}
\end{remark}

\section{Main Results}

For any $\varepsilon \in (0,\varepsilon _{2}]$, consider two state-feedback
controls in the OOCP. The first control $u_{\varepsilon, 1}(z,t)$ is
obtained from the optimal control in the PCCP (see (\ref{opt-contr-auxil}))
by replacing the matrix $P^{*}(\varepsilon)$ and the vector $h(t)$ with the
following matrix and vector, respectively:
\begin{equation}
P^{*}_{0}(\varepsilon)=\left(%
\begin{array}{l}
P_{1 0}^{*}\ \ \ \ \ \ \ \ \varepsilon P_{2 0}^{*} \\
\varepsilon \big(P_{2 0}^{*}\big)^{T}\ \ \varepsilon P_{3 0}^{*}%
\end{array}%
\right),\ \ \ \ h_{0}(t,\varepsilon)=\left(%
\begin{array}{l}
h_{1 0}(t) \\
\varepsilon h_{2 0}(t)%
\end{array}%
\right).  \label{P_0}
\end{equation}
Thus, we have
\begin{equation}
u_{\varepsilon, 1}\left(z,t\right) =-\left(G_{u}+\mathcal{E}%
\right)^{-1}B^{T}P^{*}_{0}(\varepsilon) - \left(G_{u}+\mathcal{E}%
\right)^{-1}B^{T}h_{0}(t,\varepsilon).  \label{subopt-min-str-1}
\end{equation}

In (\ref{P_0}) and (\ref{subopt-min-str-1}), $\{P_{1 0}^{*},P_{2 0}^{*},P_{3
0}^{*}\}$ and $\{h_{1 0}(t),h_{2 0}(t)\}$ are the asymptotic solutions of
the system (\ref{eq-P_1})-(\ref{eq-P_3}) and the terminal-value problem (\ref%
{eq-h_1})-(\ref{termin-cond}), respectively.

\begin{lemma}
\label{u_eps,1^o-represent}Subject to the assumptions (A1)-(A7), the
state-feedback control $u_{\varepsilon, 1}(z,t)$ can be represented in the
block form as:
\begin{equation}
u_{\varepsilon, 1}(z,t) =-\left(
\begin{array}{c}
K_{1}(\varepsilon)x+\varepsilon K_{2}(\varepsilon)y +H_{3}h_{1
0}(t)+\varepsilon H_{1}h_{2 0}(t) \\
\\
(1/\varepsilon)\Big[ \left( P_{20}^{*}\right)^{T}x+P_{30}^{*}y+h_{2 0}(t)%
\Big]%
\end{array}%
\right)  \label{derivation}
\end{equation}
where
\begin{equation}  \label{K}
K_{1}(\varepsilon)\overset{\triangle}{=}H_{3}P_{10}^{*}+\mathcal{\varepsilon
}H_{1}\left(P_{20}^{*}\right)^{T},\ \ \ K_{2}(\varepsilon)\overset{\triangle}%
{=}H_{3} P_{20}^{*}+H_{1}P_{30}^{*};
\end{equation}
$H_{1}$ and $H_{3}$ are defined in (\ref{eqH1H2});
\begin{equation}  \label{z-block}
z=\mathrm{col}\big(x,y\big),\ \ \ \ x\in E^{n-r+q},\ \ \ y\in E^{r-q}.
\end{equation}
\end{lemma}

\begin{proof}
Using the results of \cite{Glizer-Kelis2015} (proof of Lemma 8), we can
transform the matrix $(G_{u}+\mathcal{E})^{-1}B^{T}$ into the following
block-form one:
\begin{equation}
(G_{u}+\mathcal{E})^{-1}B^{T}=\left(
\begin{array}{l}
H_{3}\ \ \ \ \ \ \ \ \ \ \ \ \ \ \ \ \ \ \ \ H_{1} \\
O_{(r-q)\times (n-r+q)}\ \ \ (1/\varepsilon ^{2})I_{r-q}%
\end{array}%
\right) .  \label{block-form}
\end{equation}%
Now, the substitution of (\ref{P_0}), (\ref{z-block}) and (\ref{block-form})
into (\ref{subopt-min-str-1}) yields immediately the equation (\ref%
{derivation}).
\end{proof}

By calculation of the point-wise (with respect to $(z,t)\in E^{n}\times
[0,+\infty)$) limit of the upper block in (\ref{derivation}) for $%
\varepsilon \rightarrow 0^+$ we obtain the second state-feedback control in
the OOCP
\begin{equation}  \label{u_eps2}
u_{\varepsilon, 2}(z,t) = \left(%
\begin{array}{l}
\ \ \ \bar{u}_{1}^{\ast }(x,t) \\
\\
- (1/\varepsilon)\left[\big(P_{20}^{*}\big)^{T}x+P_{30}^{*}y+h_{2 0}(t)%
\right]%
\end{array}%
\right),
\end{equation}
where $\bar{u}_{1}^{\ast}(\cdot,\cdot)$ is given by (\ref{u_1^o*}).

\begin{lemma}
\label{admiss-u_eps_1,2}Let the assumptions (A1)-(A7) be valid. Then, there
exists a positive number $\hat{\varepsilon}$, ($\hat{\varepsilon}\le
\varepsilon_{2}$), such that for all $\varepsilon\in(0,\hat{\varepsilon}]$
the state feedback controls $u_{\varepsilon, i}(z,t)$, ($i=1,2$) are
admissible in the OOCP.
\end{lemma}

The proof is presented in Appendix B.

\begin{remark}
Due to Lemma \ref{admiss-u_eps_1,2}, the set $M_{u}$ of admissible
state-feedback controls in the OOCP in nonempty. Therefore, due to Remark %
\ref{new-problem}, the infimum $J^{*}$ of the cost functional in the OOCP
(see (\ref{inf-J})) is finite.
\end{remark}

\begin{theorem}
\label{infJ=barJ} Let the assumptions (A1)-(A7) be valid. Then, the
following equality is satisfied:
\begin{equation}  \label{J*=barJ}
J^{*} = \bar{J}^{*},
\end{equation}
where $\bar{J}^{*}$ is the optimal value of the cost functional in the ROCP
given by (\ref{bar-J*}).
\end{theorem}

\begin{proof}
We prove the theorem by contradiction. Namely, let us assume that (\ref%
{J*=barJ}) is wrong. This means the fulfilment of the inequality
\begin{equation}  \label{J*neqbarJ_0}
J^{*} \ne \bar{J}^{*}.
\end{equation}

Let us show that the inequality (\ref{J*neqbarJ_0}) yields the inequality
\begin{equation}  \label{ineq-J*<barJ}
J^{*} < \bar{J}^{*}.
\end{equation}

First of all note, that the control $u_{\varepsilon}^{*}(z,t)$ (see (\ref%
{opt-contr-auxil})), being optimal in the PCCP for all $\varepsilon \in (0,%
\hat{\varepsilon}]$, belongs to the set $M_{u}$ for these values of $%
\varepsilon$. Now, using the equations (\ref{new-perf-ind}), (\ref{inf-J}), (%
\ref{perf-ind-auxil-pr-regul}), we directly obtain the following chain of
inequalities and equalities:
\begin{equation}  \label{chain}
J^{*} \le J\big(u_{\varepsilon}^{*}(z,t)\big) \le J_{\varepsilon}\big(%
u_{\varepsilon}^{*}(z,t)\big) = J_{\varepsilon}^{*},\ \ \ \ \varepsilon \in
(0,\hat{\varepsilon}].
\end{equation}
Moreover, from the inequality (\ref{ineq-J_eps*}), we have for all $%
\varepsilon \in (0,\hat{\varepsilon}]$
\begin{equation}  \label{ineq-aux1}
\bar{J}^{*} - c\varepsilon \le J_{\varepsilon}^{*} \le \bar{J}^{*} +
c\varepsilon.
\end{equation}
Remember that $c>0$ is independent of $\varepsilon$.

The inequalities (\ref{chain})-(\ref{ineq-aux1}) imply the inequality
\begin{equation}  \label{ineq-aux2}
J^{*} \le \bar{J}^{*} + c\varepsilon,\ \ \ \ \varepsilon \in (0,\hat{%
\varepsilon}].
\end{equation}
The inequalities (\ref{J*neqbarJ_0}) and (\ref{ineq-aux2}), yield
immediately the inequality (\ref{ineq-J*<barJ}).

Since (\ref{ineq-J*<barJ}) is valid, then there exists a state-feedback
control $\tilde{u}(z,t) \in M_{u}$ such that
\begin{equation}  \label{inew-aux3}
J^{*} \le J\big(\tilde{u}(z,t)\big) < \bar{J}^{*}.
\end{equation}

Since $u_{\varepsilon}^{*}(z,t)$ is the optimal control in the PCCP, then
the following inequality is satisfied for any $\varepsilon \in (0,\hat{%
\varepsilon}]$:
\begin{equation}  \label{inew-aux4}
J_{\varepsilon}^{*} = J_{\varepsilon}\big(u_{\varepsilon}^{*}(z,t)\big) \le
J_{\varepsilon}\big(\tilde{u}(z,t)\big) = J\big(\tilde{u}(z,t)\big) +
b\varepsilon^{2},
\end{equation}
where
\begin{equation}  \label{b}
0 \le b = \int_{0}^{_\infty}\tilde{u}_{2}^{T}\big(\tilde{z}(t),t\big)\tilde{u%
}_{2}\big(\tilde{z}(t),t\big)dt < + \infty;
\end{equation}
$\tilde{u}_{2}(z,t)$ is the lower block of the dimension $r-q$ of the vector
$\tilde{u}(z,t)$; $\tilde{z}(t)$, $t\ge 0$ is the solution of (\ref{new-eq})
generated by $\tilde{u}(z,t)$.

The inequalities (\ref{ineq-aux1}) and (\ref{inew-aux4}) directly lead to
the inequality
\begin{equation}
\bar{J}^{\ast }\leq J\big(\tilde{u}(z,t)\big)+c\varepsilon +b\varepsilon
^{2},\ \ \ \ \varepsilon \in (0,\hat{\varepsilon}],  \label{inew-aux5}
\end{equation}%
which yields immediately $\bar{J}^{\ast }\leq J\big(\tilde{u}(z,t)\big)$.
The latter contradicts to the right-hand side of the inequality (\ref%
{inew-aux3}). This contradiction proves the equality (\ref{J*=barJ}). Thus,
the theorem is proven.
\end{proof}

Consider a numerical sequence $\{\varepsilon_{k}\}$ such that
\begin{equation}  \label{eps-q}
0 < \varepsilon_{k} \le \hat{\varepsilon},\ \ \ k = 1,2,...;\ \ \ \ \ \ \
\lim_{k \rightarrow + \infty}\varepsilon_{k} = 0.
\end{equation}

For the system (\ref{new-eq}), consider the following two sequences of
state-feedback controls: $\{u_{\varepsilon_{k},i}(z,t)\}$, ($i=1,2$), ($k =
1,2,...$).

\begin{theorem}
\label{minim-sequance} Let the assumptions (A1)-(A7) be valid. Then, the
following limit equations are satisfied:
\begin{equation}  \label{lim-J-baru-q}
\lim_{k \rightarrow + \infty}J\big(u_{\varepsilon_{k},i}(z,t)\big) = J^{*},\
\ \ \ i=1,2,
\end{equation}
meaning that the sequences of state-feedback controls $\{u_{%
\varepsilon_{k},i}(z,t)\}$, ($i=1,2$), ($k = 1,2,...$) are minimizing in the
OOCP (\ref{new-eq})-(\ref{new-perf-ind}).
\end{theorem}

The proof of the theorem is presented in Appendix C.

\begin{remark}
\label{connect-OOCP-ROCP}Note that the upper block of the OOCP minimizing
control sequence $\{u_{\varepsilon _{k},2}(z,t)\}$, $(k=1,2,...)$ and the
infimum of the cost functional in the OOCP coincide with the upper block of
the optimal state-feedback control and the optimal value of the cost
functional, respectively, in the ROCP (\ref{out-dyn1})-(\ref{bar-J}). The
latter control problem is regular and is of a smaller dimension than the
OOCP. Thus, in order to solve the singular OOCP, one has to solve the
smaller dimension regular ROCP, construct two gain matrices $P_{20}^{\ast }$
and $P_{30}^{\ast }$, using the equations (\ref{tilde-P_320}), and construct
the vector-valued function $h_{20}(t)$, using the equation (\ref{h_20}).
\end{remark}

\section{Example: Infinite Horizon Tracking}

Consider a body of the unit mass subject to an actuator force $u(t)$. Let $%
\tilde{x}(t)$ be the inertial position of the body. This dynamics is
described by the following differential equation and initial conditions
\begin{equation}  \label{dif-eq-example}
\frac{d^{2}\tilde{x}(t)}{dt^{2}}=u(t),\ \ \ \ t\ge 0,\ \ \ \ \tilde{x}(0)=%
\tilde{x}_{0},\ \ \ \ \frac{d\tilde{x}(0)}{dt}=\tilde{x}_{0}^{^{\prime }}.
\end{equation}
Let us denote
\begin{equation}  \label{y=dot-x}
\tilde{y}(t)\overset{\triangle}{=}\frac{d\tilde{x}(t)}{dt},\ \ \ \ t\ge 0,\
\ \ \ \ \tilde{y}_{0}\overset{\triangle}{=}\tilde{x}_{0}^{^{\prime }}.
\end{equation}
Thus, the problem (\ref{dif-eq-example}) can be rewritten as:
\begin{equation}  \label{eq-x-example}
\frac{d\tilde{x}(t)}{dt}=\tilde{y}(t),\ \ \ \ t\ge 0,\ \ \ \ \tilde{x}(0)=%
\tilde{x}_{0},
\end{equation}
\begin{equation}  \label{eq-y-example}
\frac{d\tilde{y}(t)}{dt}=u(t),\ \ \ \ t\ge 0,\ \ \ \ \tilde{y}(0)=\tilde{y}%
_{0},
\end{equation}
where $\tilde{x}(t)$ and $\tilde{y}(t)$ are the scalar state variables, $%
u(t) $ is the scalar control.

The objective of the control of the system (\ref{eq-x-example})-(\ref%
{eq-y-example}) is a trajectory tracking. Let $\tilde{x}_{\mathrm{nom}}(t)=%
\tilde{f}_{1}(t)$ and $\tilde{y}_{\mathrm{nom}}(t)=\tilde{f}_{2}(t)$, $%
t\in[0,+\infty)$, be given nominal trajectories of the body position and the
velocity, respectively. Thus, the performance index for the system's control
can be chosen as:
\begin{equation}  \label{perf-ind-example}
\tilde{J}(u)\overset{\triangle}{=} \int_{0}^{+\infty}\Big(d_{1}\big(\tilde{x}%
(t)-\tilde{f}_{1}(t)\big)^{2} + d_{2}\big(\tilde{y}(t)-\tilde{f}_{2}(t)\big)%
^{2}\Big)dt\rightarrow \min_{u(t)},
\end{equation}
where $d_{1}>0$ and $d_{2}>0$ are given constants.

\begin{remark}
Since the control $u(t)$ does not appear in the cost functional $J(u)$, the
optimal control problem (\ref{eq-x-example})-(\ref{perf-ind-example}) is
singular.
\end{remark}

Also, let us choose $\tilde{f}_{1}(t)$ and $\tilde{f}_{2}(t)$ as:
\begin{equation}  \label{tilde-f_1,2}
\tilde{f}_{1}(t)=\tilde{a}_{1}\exp(-\gamma t),\ \ \ \ \tilde{f}_{2}(t)=%
\tilde{a}_{2}\exp(-\gamma t),
\end{equation}
where $\tilde{a}_{1}>0$, $\tilde{a}_{2}>0$ and $\gamma>0$ are given
constants.

Now, we make the following transformation of the state variables in the
problem (\ref{eq-x-example})-(\ref{perf-ind-example}):
\begin{equation}  \label{trans-example}
x(t)=\tilde{x}(t)-\tilde{f}_{1}(t),\ \ \ \ y(t)=\tilde{y}(t)-\tilde{f}%
_{2}(t).
\end{equation}
Due to this transformation, we obtain a new optimal control problem,
equivalent to (\ref{eq-x-example})-(\ref{perf-ind-example}). Namely,
\begin{equation}  \label{new-eq-x-ex}
\frac{dx(t)}{dt}=y(t)+f_{1}(t),\ \ \ \ t\ge 0,\ \ \ \ x(0)=x_{0},
\end{equation}
\begin{equation}  \label{new-eq-y-ex}
\frac{dy(t)}{dt}=u(t)+f_{2}(t),\ \ \ \ t\ge 0,\ \ \ \ y(0)=y_{0},
\end{equation}
\begin{equation}  \label{perf-ind-example-2}
J(u)\overset{\triangle}{=} \int_{0}^{+\infty}\Big(d_{1}\big(x(t)\big)^{2} +
d_{2}\big(y(t)\big)^{2}\Big)dt\rightarrow \min_{u(t)},
\end{equation}
where
\begin{equation}  \label{f_1,2}
f_{1}(t)=a_{1}\exp(-\gamma t),\ \ \ \ \ f_{2}(t)=a_{2}\exp(-\gamma t),\ \ \
\ t\ge 0,
\end{equation}
\begin{equation}  \label{a_1,2}
a_{1}=\tilde{a}_{1}\gamma+\tilde{a}_{2},\ \ \ \ \ a_{2}=\tilde{a}_{2}\gamma,
\end{equation}
\begin{equation}  \label{x,y_0}
x_{0}=\tilde{x}_{0}-\tilde{a}_{1},\ \ \ \ \ y_{0}=\tilde{x}_{0}-\tilde{a}%
_{2}.
\end{equation}

In the problem (\ref{new-eq-x-ex})-(\ref{perf-ind-example-2}), the nominal
trajectory is $\{x_{\mathrm{nom}}(t)\equiv 0 , y_{\mathrm{nom}}(t)\equiv 0\}$%
.

Based on the results of the previous sections, proceed to solution of (\ref%
{new-eq-x-ex})-(\ref{perf-ind-example-2}). First of all, let us note that
this problem is a particular case of the OOCP (\ref{new-eq})-(\ref%
{new-perf-ind}) for $n=2$, $r=1$, $q=0$, and
\begin{equation}  \label{example-data-1}
A_{1}=0,\ \ \ A_{2}=1,\ \ \ A_{3}=0,\ \ A_{4}=0,\ \ \ B_{1}=0,\ \ \ B_{2}=1,
\end{equation}
\begin{equation}  \label{example-data-2}
D_{1}=d_{1},\ \ \ D_{2}=d_{2},\ \ \ G=0.
\end{equation}

For the problem (\ref{new-eq-x-ex})-(\ref{perf-ind-example-2}), $S_{0}=\left(%
\frac{1}{d_{2}}+1\right)$, and the equation (\ref{eq-P_10^o-new-form})
becomes
\begin{equation}  \label{eq-P_10-ex}
-\left(\frac{1}{d_{2}}+1\right)\big(P_{1 0}\big)^{2}+d_{1}=0.
\end{equation}
The matrices $\bar{B}$ and $F_{1}$ (see (\ref{B^o}) and (\ref{F_1^TF_1}))
are the scalars $\bar{B}=A_{2}=1$ and $F_{1}=1/\sqrt{d_{1}}>0$, meaning that
the assumptions (A6) and (A7) are valid. The equation (\ref{eq-P_10-ex}) has
two solutions, positive and negative. We choose the positive solution of
this equation, i.e.,
\begin{equation}  \label{P_10*-ex}
P_{1 0}^{*}=\sqrt{\frac{d_{1}d_{2}}{1+d_{2}}}.
\end{equation}
Due to this equation and the equation (\ref{tilde-P_320}), we have
\begin{equation}  \label{P_230*-ex}
P_{2 0}^{*}=\sqrt{\frac{d_{1}}{1+d_{2}}},\ \ \ \ \ P_{3 0}^{*}=\sqrt{d_{2}}.
\end{equation}
Using (\ref{matr-tildA_10}), we obtain
\begin{equation}  \label{mathcal A_0}
{\mathcal{A}}_{0}=-\sqrt{\frac{d_{1}(1+d_{2})}{d_{2}}}<0.
\end{equation}

Now, let us obtain $h_{1 0}(t)$, $h_{2 0}(t)$ and $s_{0}(t)$. Using the
equations (\ref{h_20}), (\ref{h_10}) and (\ref{s_0}), as well as the data of
the example (\ref{f_1,2}), (\ref{example-data-1}-(\ref{example-data-2}), we
obtain
\begin{equation}  \label{h_10-exmple}
h_{1 0}(t)=\frac{a_{1}P_{1 0}^{*}}{\gamma-{\mathcal{A}}_{0}}\exp(-\gamma
t),\ \ \ \ t \ge 0,
\end{equation}
\begin{equation}  \label{h_20-exmple}
h_{2 0}(t)=\frac{a_{1}P_{2 0}^{*}}{\gamma-{\mathcal{A}}_{0}}\exp(-\gamma
t),\ \ \ \ t \ge 0,
\end{equation}
\begin{equation}  \label{s_0-ex}
s_{0}(t)=\frac{a_{1}^{2}P_{1 0}^{*}(2\gamma-{\mathcal{A}}_{0})}{%
2\gamma(\gamma-{\mathcal{A}}_{0})^{2}}\exp(-2\gamma t),\ \ \ \ t \ge 0.
\end{equation}

Due to Theorem \ref{infJ=barJ}, and the equations (\ref{bar-J*}), (\ref%
{h_10-exmple}) and (\ref{s_0-ex}), the infimum of the cost functional $J(u)$
in the optimal control problem (\ref{new-eq-x-ex})-(\ref{perf-ind-example-2}%
) is
\begin{equation}
\inf_{u}J(u)=P_{10}^{\ast }\left[ \left( x_{0}+\frac{a_{1}}{\gamma -{%
\mathcal{A}}_{0}}\right) ^{2}-\frac{a_{1}^{2}{\mathcal{A}}_{0}}{2\gamma
(\gamma -{\mathcal{A}}_{0})^{2}}\right] >0.  \label{inf-ex}
\end{equation}%
Finally, using Theorem \ref{minim-sequance}, as well as the equation (\ref%
{derivation}) and the fact that $q=0$, yields the minimizing sequence in the
optimal control problem (\ref{new-eq-x-ex})-(\ref{perf-ind-example-2})
\begin{equation}
u_{\varepsilon _{k}}(z,t)=-\frac{1}{\varepsilon _{k}}\big(P_{20}^{\ast
}x+P_{30}^{\ast }y+h_{20}(t)\big),\ \ \ \ z=\mathrm{col}(x,y),\ \ \
k=1,2,....  \label{minim-seq-ex}
\end{equation}

\begin{figure}
[H]
\begin{center}
\includegraphics[width=\textwidth]{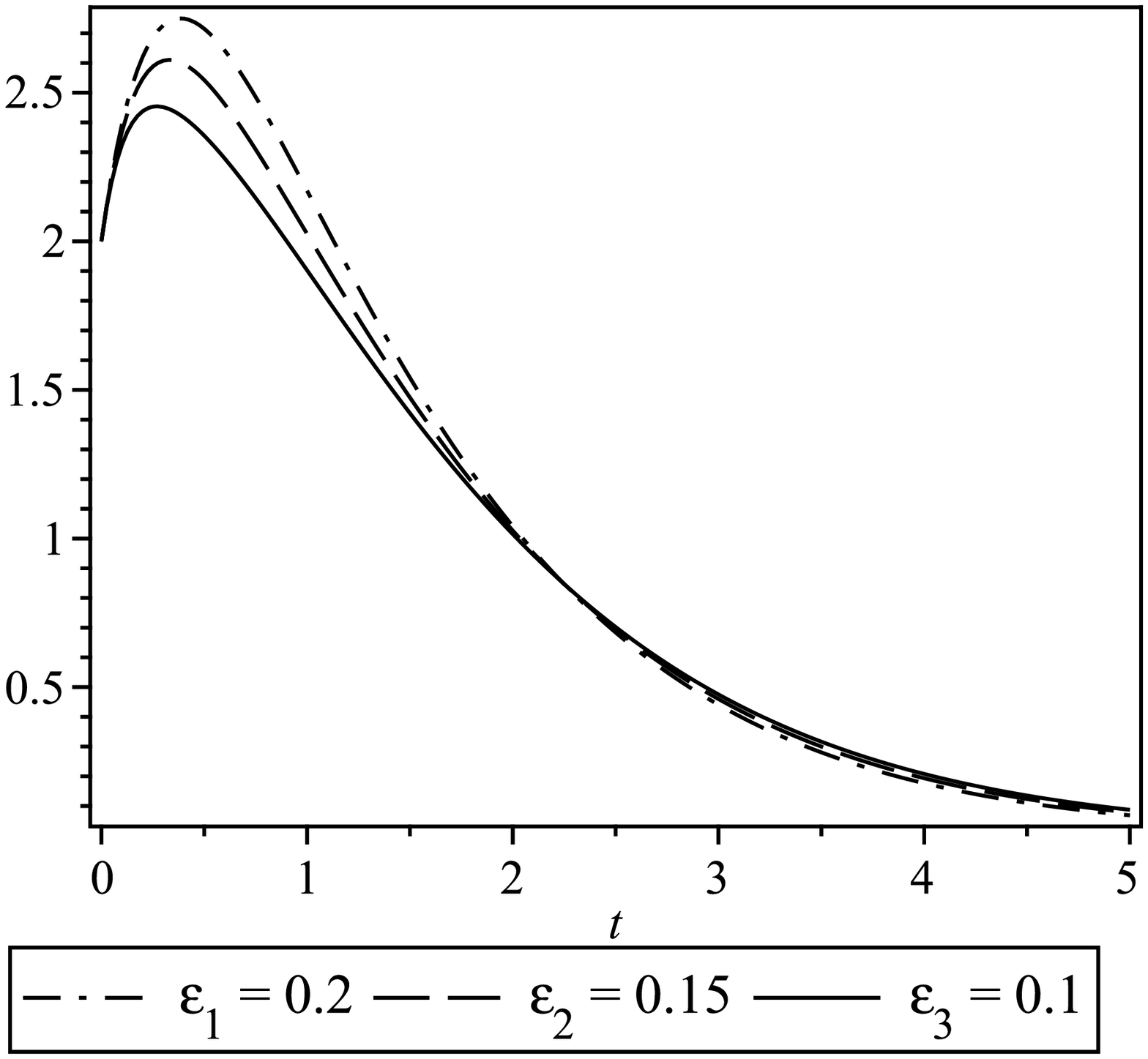}
\end{center}
\caption{$x$-component of the trajectory of the system
(\ref{new-eq-x-ex})-(\ref{new-eq-y-ex})} \label{x-compon1}
\end{figure}

\begin{figure}
[H]
\begin{center}
\includegraphics[width=\textwidth]{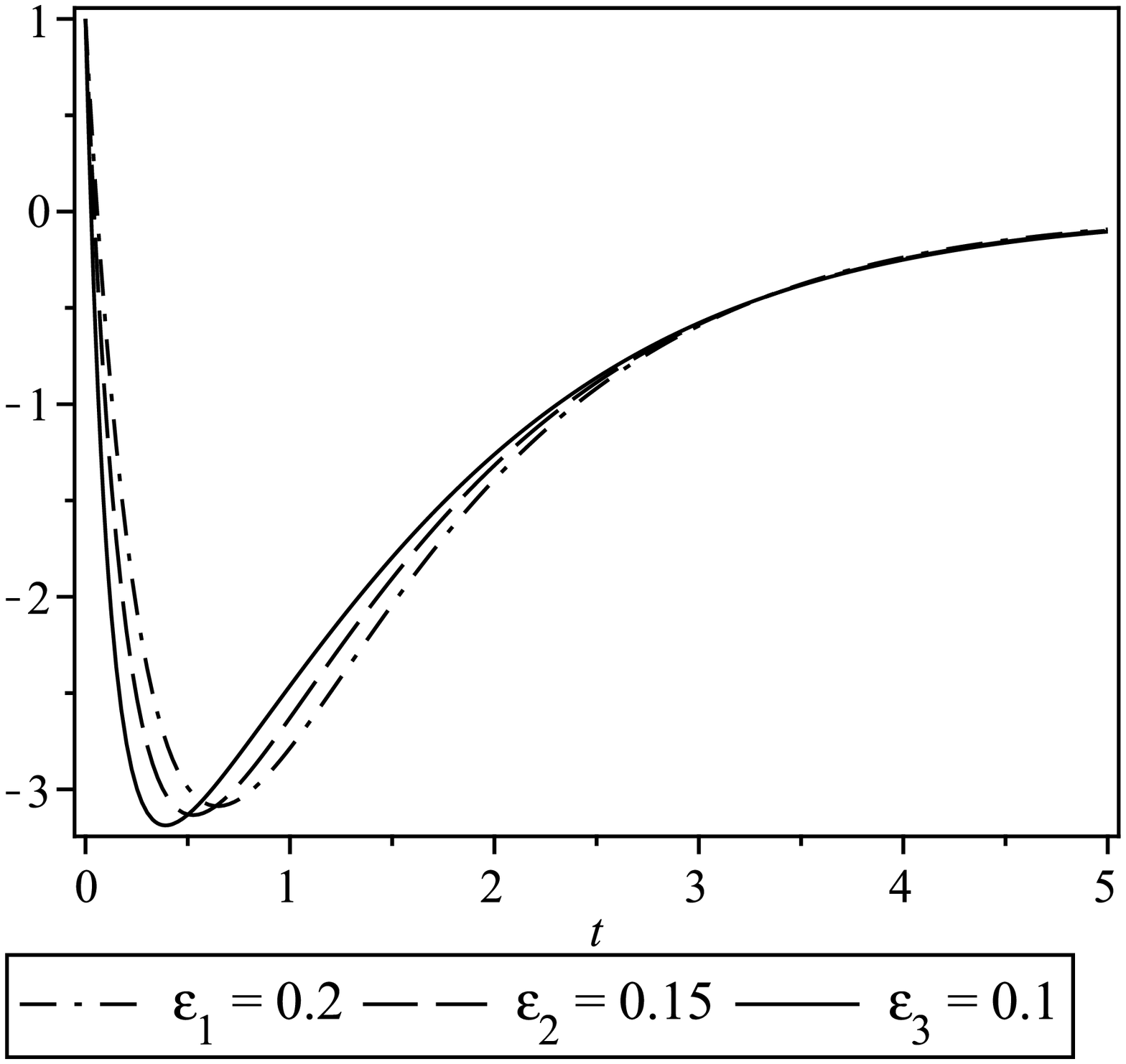}
\end{center}
\caption{$y$-component of the trajectory of the system
(\ref{new-eq-x-ex})-(\ref{new-eq-y-ex})} \label{y-compon1}
\end{figure}

\begin{figure}
[H]
\begin{center}
\includegraphics[width=\textwidth] {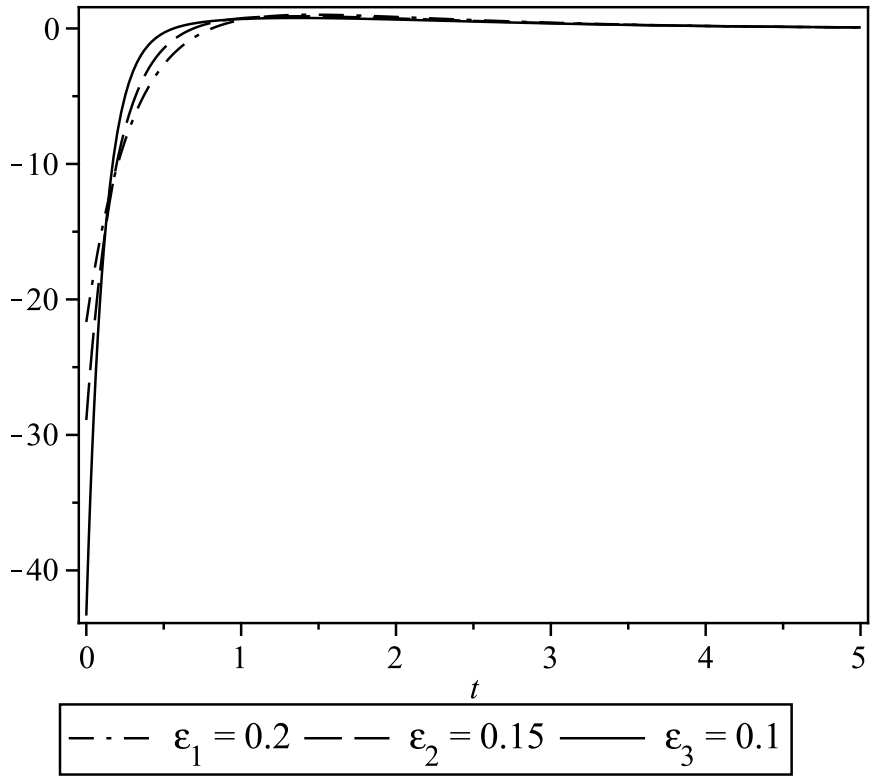}
\end{center}
\caption{Time realization of $u_{\varepsilon_{k}}(z,t)$}
\label{u-time-real1}
\end{figure}

In Figs. \ref{x-compon1} and \ref{y-compon1}, the $x$- and
$y$-components of the trajectory of the system
(\ref{new-eq-x-ex})-(\ref{new-eq-y-ex}),
generated by the control (\ref{minim-seq-ex}) with different values of $%
\varepsilon _{k}$, are depicted for the following numerical data: $a_{1}=4$,
$a_{2}=2$, $\gamma =1$, $x_{0}=2$, $y_{0}=1$, $d_{1}=2$, $d_{2}=1$. It is
seen that both components tend to the components of the nominal trajectory
for $t\rightarrow +\infty $. Moreover, for the smaller $\varepsilon _{k}$,
the rate of this convergence is larger. In Fig. \ref{u-time-real1}, the time
realization of the control (\ref{minim-seq-ex}) along the trajectory of the
system (\ref{new-eq-x-ex})-(\ref{new-eq-y-ex}), depicted in Figs. \ref%
{x-compon1}--\ref{y-compon1}, is presented. It is seen that for $\varepsilon
_{k}$ approaching zero, this time realization tends to an impulse-like
function with the impulse at $t=0$.

\section{Concluding Remarks}

\textbf{CRI.} In this paper, an infinite horizon linear-quadratic optimal
control problem for a system with known time-varying additive disturbance is
considered. A weight matrix of the control cost in the cost functional of
this problem is singular but, in general, non-zero. Due to this singularity
of the weight matrix, the optimal control problem itself is singular.
However, if the weight matrix is non-zero, only a part of the coordinates of
the control is singular, while the others are regular.\newline
\textbf{CRII.} Subject to proper assumptions, the linear system of the
control problem is transformed equivalently to a new system consisting of
three modes. The first mode is uncontrolled directly (i.e. it does not
contain the control at all), the second mode is controlled directly only by
the regular coordinates of the control, while the third mode is controlled
directly by the entire control. Due to this transformation, a new control
problem, equivalent to the initially formulated one, is obtained. This new
singular optimal control problem is considered as an original one in the
paper.\newline
\textbf{CRIII.} The original control problem is solved by a regularization
approach, i.e., by its approximate transformation to an auxiliary regular
optimal control problem. The latter has the same equation of dynamics and a
similar cost functional augmented by an infinite horizon integral of the
squares of the singular control coordinates with a small positive weight.
Hence, the auxiliary problem is an infinite horizon linear-quadratic optimal
control problem with partial cheap control. An asymptotic analysis of this
problem is carried out.\newline
\textbf{CRIV.} Based on this asymptotic analysis, it is shown that the
infimum of the cost functional in the original (singular) optimal control
problem is finite. The explicit expression of this infimum is derived. The
minimizing state-feedback control sequence in the original problem also is
designed. Some coordinates of the minimizing sequence are convergent in the
class of regular functions. Namely, the coordinates of the minimizing
sequence, corresponding to the regular control coordinates, are point-wise
convergent in this class of functions. The corresponding limits constitute
the regular part of the optimal state-feedback control in the original
problem.\newline
\textbf{CRV.} It is shown that the infimum of the cost functional and the
regular part of the optimal state-feedback control in the original singular
problem coincide with the optimal value of the cost functional and the upper
block of the optimal state-feedback control, respectively, in a reduced
dimension regular optimal control problem (reduced control problem). The
dimension of the upper block of the optimal state-feedback control in the
reduced control problem equals to the number of the regular coordinates of
the control in the original problem. The reduced control problem is
connected with the zero-order asymptotic solutions of the equations, arising
in the optimality conditions of the auxiliary partial cheap control problem.%
\newline
\textbf{CRVI.} Using the obtained theoretical results, the problem of
singular infinite horizon trajectory tracking with two scalar state
variables and a scalar control is solved. Numerical simulation shows that
the state-feedback controls of the minimizing sequence generate
trajectories, approaching well enough the nominal ones. The
time-realizations of these state-feedback controls tend to an impulse-like
function with the impulse at $t=0$.

\section{Appendix A: Proof of Lemma \protect\ref{asympt-h}}

\subsection{Auxiliary results}

Let for a given $\varepsilon>0$, the matrix-valued function $%
\Psi(t,\varepsilon)$, $t\ge 0$ of the dimension $(m_{1}+m_{2})$ be the
solution of the following initial-value problem:
\begin{equation}  \label{auxil-eq-Psi}
\frac{d\Psi(t,\varepsilon)}{dt}=C(\varepsilon)\Psi(t,\varepsilon),\ \ \ \
\Psi(0,\varepsilon)=I_{m_{1}+m_{2}},
\end{equation}
where a given $(m_{1}+m_{2})$-matrix $C(\varepsilon)$ has the block form
\begin{equation}  \label{C-block}
C(\varepsilon)=\left(%
\begin{array}{l}
C_{1}(\varepsilon)\ \ \ \ \ \ \ \ \ \ C_{2}(\varepsilon) \\
(1/\varepsilon)C_{3}(\varepsilon)\ \ \ (1/\varepsilon)C_{4}(\varepsilon)%
\end{array}%
\right),
\end{equation}
the blocks $C_{1}(\varepsilon)$, $C_{2}(\varepsilon)$, $C_{3}(\varepsilon)$
and $C_{4}(\varepsilon)$ are of the dimensions $m_{1}\times m_{1}$, $%
m_{1}\times m_{2}$, $m_{2}\times m_{1}$ and $m_{2}\times m_{2}$,
respectively.

Represent the matrix $\Psi(t,\varepsilon)$ in the block form
\begin{equation}  \label{block-Psi}
\Psi(t,\varepsilon)=\left(%
\begin{array}{l}
\Psi_{1}(t,\varepsilon)\ \ \ \ \Psi_{2}(t,\varepsilon) \\
\Psi_{3}(t,\varepsilon)\ \ \ \ \Psi_{4}(t,\varepsilon)%
\end{array}%
\right),
\end{equation}
where the blocks $\Psi_{1}(t,\varepsilon)$, $\Psi_{2}(t,\varepsilon)$, $%
\Psi_{3}(t,\varepsilon)$, and $\Psi_{4}(t,\varepsilon)$ are of the
dimensions $m_{1}\times m_{1}$, $m_{1}\times m_{2}$, $m_{2}\times m_{1}$ and
$m_{2}\times m_{2}$, respectively.

\begin{proposition}
\label{block-est-Psi}Let there exists a positive number $\check{\varepsilon}$
such that the matrices $C_{j}(\varepsilon)$, $(j=1,..4)$ are continuous with
respect to $\varepsilon\in[0,\check{\varepsilon}]$. Let, there exists a
positive number $\omega$ such that all the eigenvalues $\lambda\Big(%
C_{4}(\varepsilon)\Big)$ of the matrix $C_{4}(\varepsilon)$ satisfy the
inequality
\begin{equation}  \label{eigenval-ineq-3}
\mathrm{Re}\lambda\Big(C_{4}(\varepsilon)\Big) < -\omega,\ \ \ \ \varepsilon%
\in[0,\check{{\varepsilon}}].
\end{equation}
Let, there exists a positive number $\kappa$ such that all the eigenvalues $%
\lambda\Big(\bar{C}(\varepsilon)\Big)$ of the matrix $\bar{C}(\varepsilon)%
\overset{\triangle}{=}C_{1}(\varepsilon)-C_{2}(\varepsilon)C_{4}^{-
1}(\varepsilon)C_{3}(\varepsilon)$ satisfy the inequality
\begin{equation}  \label{eigenval-ineq-4}
\mathrm{Re}\lambda\Big(\bar{C}(\varepsilon)\Big) < -\kappa,\ \ \ \
\varepsilon\in[0,\check{{\varepsilon}}].
\end{equation}
Then, there exists a positive number $\bar{\varepsilon}$, ($\bar{\varepsilon}%
\le \check{{\varepsilon}})$, such that for all $\varepsilon \in (0,\bar{%
\varepsilon}]$, the following inequalities are satisfied:
\begin{equation}  \label{ineq-Psi-1-3}
\big\|\Psi_{i}\big(t,\varepsilon\big)\big\|\le a\exp\big(-\kappa t\big),\ \
\ \ i=1,3,\ \ \ \ 0\le t < +\infty,
\end{equation}
\begin{equation}  \label{ineq-Psi-2}
\big\|\Psi_{2}\big(t,\varepsilon\big)\big\|\le a\varepsilon\exp\big(-\kappa t%
\big),\ \ \ \ 0\le t < +\infty,
\end{equation}
\begin{equation}  \label{ineq-Psi-4}
\big\|\Psi_{4}\big(t,\varepsilon\big)\big\|\le a\Big(\varepsilon\exp\big(%
-\kappa t\big)+\exp\big(-\omega t/\varepsilon\big)\Big),\ \ \ \ 0\le t <
+\infty,
\end{equation}
where $a>0$ is some constant independent of $\varepsilon$.
\end{proposition}

\begin{proof}
The statement of the proposition directly follows from the results of \cite%
{Gl-2003} (Theorem 2.3).
\end{proof}

Now, let us set $m_{1}=n-r+q$, $m_{2}=r-q$ and
\begin{equation}  \label{C_j}
C_{1}(\varepsilon)={\mathcal{A}}_{1}^{T}(\varepsilon),\ \ C_{2}(\varepsilon)=%
{\mathcal{A}}_{3}^{T}(\varepsilon),\ \ C_{3}(\varepsilon)={\mathcal{A}}%
_{2}^{T}(\varepsilon),\ \ C_{4}(\varepsilon)={\mathcal{A}}%
_{4}^{T}(\varepsilon).
\end{equation}

\begin{corollary}
\label{block-est-Phi}Let the assumptions (A1)-(A3), (A5)-(A7) be valid.
Then, there exists a positive number $\bar{\varepsilon}_{0}$, ($\bar{%
\varepsilon}_{0}\le\varepsilon_{0}$), such that for all $\varepsilon\in(0,%
\bar{\varepsilon}_{0}]$ the following inequalities are satisfied:
\begin{equation}  \label{ineq-Phi-1-3}
\big\|\Psi_{i}\big(t,\varepsilon\big)\big\|\le a\exp\big(-\alpha t\big),\ \
\ \ i=1,3,\ \ \ \ 0\le t < +\infty,
\end{equation}
\begin{equation}  \label{ineq-Phi-2}
\big\|\Psi_{2}\big(t,\varepsilon\big)\big\|\le a\varepsilon\exp\big(-\alpha t%
\big),\ \ \ \ 0\le t < +\infty,
\end{equation}
\begin{equation}  \label{ineq-Phi-4}
\big\|\Psi_{4}\big(t,\varepsilon\big)\big\|\le a\Big(\varepsilon\exp\big(%
-\alpha t\big)+\exp\big(-\beta t/\varepsilon\big)\Big),\ \ \ \ 0\le t <
+\infty,
\end{equation}
where $a>0$ is some constant independent of $\varepsilon$.
\end{corollary}

\begin{proof}
First of all note that, due to the equations (\ref{mathcal-A_1})-(\ref%
{mathcal-A_4}) and Lemma \ref{justif-asymp-sol-Riccati}, the matrices $%
C_{j}(\varepsilon)$, $(j=1,..4)$, given by (\ref{C_j}), are continuous with
respect to $\varepsilon\in[0,\varepsilon_{0}]$.

Using the equations (\ref{eqL-1}), (\ref{tilde-P_320}), (\ref{mathcal-A_1})-(%
\ref{mathcal-A_4}), (\ref{C_j}) and Lemma \ref{justif-asymp-sol-Riccati}, we
obtain
\begin{equation}  \label{C_1-new}
C_{1}(\varepsilon)=A_{1}^{T}-P_{1 0}^{*}S_{1}+\Delta C_{1}(\varepsilon),
\end{equation}
\begin{equation}  \label{C_2-new}
C_{2}(\varepsilon)=-P_{1 0}^{*}A_{2}\big(D_{2}\big)^{-1/2}+\Delta
C_{2}(\varepsilon),
\end{equation}
\begin{equation}  \label{C_3-new}
C_{3}(\varepsilon)=A_{2}^{T}+\Delta C_{3}(\varepsilon),
\end{equation}
\begin{equation}  \label{C_4-new}
C_{4}(\varepsilon)=-\big(D_{2}\big)^{1/2}+\Delta C_{4}(\varepsilon),
\end{equation}
where $\Delta C_{j}(\varepsilon)$, ($j=1,...,4$) are some matrices
satisfying the inequalities
\begin{equation}  \label{Delta-mathcal-ineq}
\big\|\Delta C_{j}(\varepsilon)\big\|\le a\varepsilon,\ \ \ j=1,...,4.\ \ \
\varepsilon\in[0,\varepsilon_{0}],
\end{equation}
$a>0$ is some constant independent of $\varepsilon$.

The equation (\ref{C_4-new}) and the inequalities (\ref{eigenval-ineq-1}), (%
\ref{Delta-mathcal-ineq}) directly yield the existence of a positive number $%
\bar{\varepsilon}_{1}$, ($\bar{\varepsilon}_{1}\le \varepsilon_{0}$), such
that all the eigenvalues $\lambda\Big(C_{4}(\varepsilon)\Big)$ of the matrix
$C_{4}(\varepsilon)$ satisfy the inequality
\begin{equation}  \label{eigenval-ineq-1-1}
\mathrm{Re}\lambda\Big(C_{4}(\varepsilon)\Big)<-\beta,\ \ \ \ \varepsilon\in[%
0,\bar{\varepsilon}_{1}].
\end{equation}
Now, based on (\ref{C_1-new})-(\ref{eigenval-ineq-1-1}) and (\ref%
{matr-tildA_10}), we immediately obtain that the matrix $\bar{C}%
(\varepsilon)=C_{1}(\varepsilon) -
C_{2}(\varepsilon)C_{4}^{-1}(\varepsilon)C_{3}(\varepsilon)$ can be
represented as:
\begin{equation}  \label{bar-C}
\bar{C}(\varepsilon)={\mathcal{A}}_{0}^{T}+\overline{\Delta C}(\varepsilon),
\end{equation}
where $\overline{\Delta C}(\varepsilon)$ is some matrix satisfying the
inequality
\begin{equation}  \label{bar-Delta-C-ineq}
\big\|\overline{\Delta C}(\varepsilon)\big\|\le a\varepsilon,\ \ \ \
\varepsilon\in[0,\bar{\varepsilon}_{1}].
\end{equation}

From the equation (\ref{bar-C}), and the inequalities (\ref{eigenval-ineq-2}%
) and (\ref{bar-Delta-C-ineq}), we directly obtain the existence of a
positive number $\bar{\varepsilon}_{2}$, ($\bar{\varepsilon}_{2}\leq \bar{%
\varepsilon}_{1}$), such that all the eigenvalues $\lambda \Big(\bar{C}%
(\varepsilon )\Big)$ of the matrix $\bar{C}(\varepsilon )$ satisfy the
inequality
\begin{equation}
\mathrm{Re}\lambda \Big(\bar{C}(\varepsilon )\Big)<-\alpha ,\ \ \ \
\varepsilon \in \lbrack 0,\bar{\varepsilon}_{2}].  \label{eigenval-ineq-2-1}
\end{equation}%
Thus, we have shown that the blocks (\ref{C_j}) of the matrix $C(\varepsilon
)$ satisfy all the conditions of Proposition \ref{block-est-Psi}, which
completes the proof of the corollary.
\end{proof}

\subsection{Main part of the proof}

Let $\Delta_{1}(t,\varepsilon)$ and $\Delta_{2}(t,\varepsilon)$ be vectors,
defined as follows:
\begin{equation}  \label{Delta_1,2}
\Delta_{1}(t,\varepsilon)\overset{\triangle}{=}h_{1}(t,\varepsilon)-h_{1
0}(t),\ \ \ \Delta_{2}(t,\varepsilon)\overset{\triangle}{=}%
h_{2}(t,\varepsilon)-h_{2 0}(t),\ \ \ t\ge 0.
\end{equation}

Substitution of (\ref{Delta_1,2}) into (\ref{eq-h_1})-(\ref{termin-cond}),
and using (\ref{eq-h_10})-(\ref{eq-h_20}), (\ref{term-cond-h_10}) and (\ref%
{term-cond-h_20}) yield the problem for $\Delta_{1}(t,\varepsilon)$ and $%
\Delta_{2}(t,\varepsilon)$
\begin{equation}  \label{eq-Del_1}
\frac{d\Delta_{1}(t,\varepsilon)}{dt}=-{\mathcal{A}}_{1}^{T}(\varepsilon)%
\Delta_{1}(t,\varepsilon)-{\mathcal{A}}_{3}^{T}(\varepsilon)\Delta_{2}(t,%
\varepsilon)+ \Gamma_{1}(t,\varepsilon),
\end{equation}
\begin{equation}  \label{eq-Del_2}
\varepsilon\frac{d\Delta_{2}(t,\varepsilon)}{dt}=-{\mathcal{A}}%
_{2}^{T}(\varepsilon)\Delta_{1}(t,\varepsilon) - {\mathcal{A}}%
_{4}^{T}(\varepsilon)\Delta_{2}(t,\varepsilon)+ \Gamma_{2}(t,\varepsilon),
\end{equation}
\begin{equation}  \label{termin-cond-Delt}
\Delta_{1}(+\infty,\varepsilon)=0,\ \ \ \ \
\Delta_{2}(+\infty,\varepsilon)=0,
\end{equation}
where
\begin{eqnarray}
\Gamma_{1}(t,\varepsilon) = \big({\mathcal{A}}_{1}^{T}(0)-{\mathcal{A}}%
_{1}^{T}(\varepsilon)\big)h_{1 0}(t)+\big({\mathcal{A}}_{3}^{T}(0)-{\mathcal{%
A}}_{3}^{T}(\varepsilon)\big)h_{2 0}(t)  \notag \\
+ \big(P_{1 0}^{*}-P_{1}^{*}(\varepsilon)\big)f_{1}(t) -\varepsilon
P_{2}^{*}(\varepsilon)f_{2}(t),  \label{Gamma_1}
\end{eqnarray}
\begin{eqnarray}
\Gamma_{2}(t,\varepsilon) = \big({\mathcal{A}}_{2}^{T}(0)-{\mathcal{A}}%
_{2}^{T}(\varepsilon)\big)h_{1 0}(t)+\big({\mathcal{A}}_{4}^{T}(0)-{\mathcal{%
A}}_{4}^{T}(\varepsilon)\big)h_{2 0}(t)  \notag \\
- \varepsilon\big(P_{2}^{*}(\varepsilon)\big)^{T}f_{1}(t) -\varepsilon
P_{3}^{*}(\varepsilon)f_{2}(t).  \label{Gamma_2}
\end{eqnarray}

Using the equations (\ref{eqL-1})-(\ref{eqH1H2}), (\ref{mathcal-A_1})-(\ref%
{mathcal-A_4}), Lemma \ref{justif-asymp-sol-Riccati} (the inequalities (\ref%
{ineq-tilde(P-P_0)})), and the inequalities (\ref{f-ineq}), (\ref{ineq-h_10}%
), (\ref{ineq-h_20}), we directly have
\begin{equation}  \label{ineq-Gamma_1,2}
\big\|\Gamma_{i}(t,\varepsilon)\big\|\le a\varepsilon\exp(-\gamma t),\ \ \ \
i=1,2,\ \ \ \ t\ge 0,\ \ \ \varepsilon \in (0,\tilde{\varepsilon}_{0}],
\end{equation}
where $\tilde{\varepsilon}_{0}\overset{\triangle}{=}\min\{1,\varepsilon_{0}%
\} $; $a>0$ is some constant independent of $\varepsilon$.

Using the equations (\ref{auxil-eq-Psi})-(\ref{block-Psi}) and (\ref{C_j}),
we can represent the solution of the problem (\ref{eq-Del_1})-(\ref%
{termin-cond-Delt}) as follows:
\begin{equation}  \label{Delt-1}
\Delta_{1}(t,\varepsilon)=\int_{0}^{+\infty}\Big(\Psi_{1}(\sigma,%
\varepsilon)\Gamma_{1}(\sigma+t,\varepsilon) +
(1/\varepsilon)\Psi_{2}(\sigma,\varepsilon)\Gamma_{2}(\sigma+t,\varepsilon)%
\Big)d\sigma,\ t\ge 0,
\end{equation}
\begin{equation}  \label{Delt-2}
\Delta_{2}(t,\varepsilon)=\int_{0}^{+\infty}\Big(\Psi_{3}(\sigma,%
\varepsilon)\Gamma_{1}(\sigma+t,\varepsilon) +
(1/\varepsilon)\Psi_{4}(\sigma,\varepsilon)\Gamma_{2}(\sigma+t,\varepsilon)%
\Big)d\sigma,\ t\ge 0.
\end{equation}
These equations, along with the inequalities (\ref{ineq-Phi-1-3})-(\ref%
{ineq-Phi-4}) and (\ref{ineq-Gamma_1,2}), directly yield the inequalities
\begin{equation}  \label{ineq-Delta}
\big\|\Delta_{i}(t,\varepsilon)\big\|\le c\varepsilon\exp(-\mu t),\ \ \ \
i=1,2,\ \ \ t\ge 0,\ \ \ \varepsilon\in(0,\varepsilon_{1}],
\end{equation}
where $\varepsilon_{1}=\min\{\bar{\varepsilon}_{0} , \tilde{\varepsilon}%
_{0}\}$, $c>0$ is some constant independent of $\varepsilon$.

The equation (\ref{Delta_1,2}) and the inequalities (\ref{ineq-Delta})
immediately imply the inequalities (\ref{ineq-h_1}). This completes the
proof of the lemma.

\section{Appendix B: Proof of Lemma \protect\ref{admiss-u_eps_1,2}}

\label{App-B} We prove the lemma for $u_{\varepsilon,1}(z,t)$. The
admissibility of $u_{\varepsilon,2}(z,t)$ is shown similarly.

Substitution of $u_{\varepsilon,1}\left(z,t\right)$ into (\ref{new-eq}), and
using the equations (\ref{S_u}), (\ref{S_u-new-form})-(\ref{eqH1H2}), (\ref%
{A-block-form}), (\ref{f-block}), (\ref{z_0-blocks}) and (\ref{z-block})
yield the following initial-value problem in the interval $t\in[0,+\infty)$:
\begin{equation}  \label{eq-x-closed}
\frac{dx(t)}{dt}={\mathcal{A}}_{1 0}(\varepsilon)x(t)+{\mathcal{A}}_{2
0}(\varepsilon)y(t)-S_{1}h_{1 0}(t)-\varepsilon S_{2}h_{2 0}(t)+f_{1}(t),\ \
\ x(0)=x_{0},
\end{equation}
\begin{eqnarray}
\varepsilon\frac{dy(t)}{dt}={\mathcal{A}}_{3 0}(\varepsilon)x(t) +{\mathcal{A%
}}_{4 0}(\varepsilon)y(t)  \notag \\
-\varepsilon S_{2}^{T}h_{1 0}(t)-S_{3}(\varepsilon)h_{2 0}(t)+\varepsilon
f_{2}(t),\ \ \ y(0)=y_{0},  \label{eq-y-closed}
\end{eqnarray}
where
\begin{equation}  \label{mathcal-A_10}
{\mathcal{A}}_{1 0}(\varepsilon)=A_{1}- S_{1}P^{*}_{1 0}- \varepsilon S_{2}%
\big(P^{*}_{2 0}\big)^{T},
\end{equation}
\begin{equation}  \label{mathcal-A_20}
{\mathcal{A}}_{2 0}(\varepsilon)=A_{2}-\varepsilon S_{1}P^{*}_{2 0}-
\varepsilon S_{2}P^{*}_{3 0},
\end{equation}
\begin{equation}  \label{mathcal-A_30}
{\mathcal{A}}_{3 0}(\varepsilon)=\varepsilon A_{3}- \varepsilon
S_{2}P^{*}_{1 0}-S_{3}(\varepsilon)\big(P^{*}_{2 0}\big)^{T},
\end{equation}
\begin{equation}  \label{mathcal-A_40}
{\mathcal{A}}_{4 0}(\varepsilon)=\varepsilon
A_{4}-\varepsilon^{2}S_{2}^{T}P^{*}_{2 0}- S_{3}(\varepsilon)P^{*}_{3 0}.
\end{equation}

For any $\varepsilon\in(0,\varepsilon_{0}]$, the problem (\ref{eq-x-closed}%
)-(\ref{eq-y-closed}) has the unique locally absolutely continuous solution $%
z(t,\varepsilon)=\{x(t,\varepsilon) , y(t,\varepsilon)\}$, $t\in [0,+\infty)$%
.

Using the inequalities (\ref{f-ineq}), (\ref{ineq-h_10}), (\ref{ineq-h_20}),
one can obtain (similarly to the inequalities (\ref{ineq-Delta})) the
following inequalities:
\begin{equation}  \label{ineq-x-y}
\|x(t,\varepsilon)\|\le a\exp(-\mu t),\ \ \ \ \|y(t,\varepsilon)\|\le
a\exp(-\mu t),\ \ \ t\ge 0,\ \ \ \ \varepsilon\in(0,\check{{\varepsilon}}%
_{1}],
\end{equation}
where $\check{{\varepsilon}}_{1}>0$, ($\check{{\varepsilon}}_{1}\le
\varepsilon_{2}$) is some constant; $a>0$ is some constant independent of $%
\varepsilon$; the constant $\mu$ is given by (\ref{mu}).

Due to (\ref{ineq-x-y}), $z(t,\varepsilon)\in L^{2}[0,+\infty; E^{n}]$ for
all $\varepsilon\in(0,\check{{\varepsilon}}_{1}]$. The latter inclusion,
along with the equation (\ref{subopt-min-str-1}), yields the inclusion $%
u_{\varepsilon,1}\big(z(t,\varepsilon),t\big)\in L^{2}[0,+\infty; E^{r}]$
for all $\varepsilon\in(0,\check{{\varepsilon}}_{1}]$. Thus, the
state-feedback control $u_{\varepsilon,1}(z,t)$ satisfies all the conditions
of the admissibility in the OOCP, which completes the proof of the lemma.

\section{Appendix C: Proof of Theorem \protect\ref{minim-sequance}}

We prove the theorem for the sequence $\{u_{\varepsilon_{k},1}(z,t)\}$, ($%
k=1,2,...$). The statement of the theorem with respect to the sequence $%
\{u_{\varepsilon_{k},2}(z,t)\}$, ($k=1,2,...$) is proven similarly.

\subsection{Auxiliary results}

Similarly to proof of Lemma \ref{admiss-u_eps_1,2} (see Section \ref{App-B}%
), the substitution of $u_{\varepsilon,1}\left(z,t\right)$ into (\ref{new-eq}%
) yields the initial-value problem (\ref{eq-x-closed})-(\ref{eq-y-closed})
in the interval $t\in[0,+\infty)$. Let us construct the zero-order
asymptotic solution to this problem. Following the Boundary Function Method
\cite{Vasil'eva}, we look for this asymptotic solution in the form
\begin{equation}  \label{asymp-sol}
x_{0}^{\mathrm{as}}(t,\varepsilon)=x_{0}^{o}(t)+x_{0}^{b}(\tau),\ \ \ \
y_{0}^{\mathrm{as}}(t,\varepsilon)=y_{0}^{o}(t)+y_{0}^{b}(\tau),\ \ \
\tau=t/\varepsilon,
\end{equation}
where $\{x_{0}^{o}(t) , y_{0}^{o}(t)\}$ is the so-called outer solution, $%
x_{0}^{b}(\tau)$ and $y_{0}^{b}(\tau)$ are the boundary correction terms.

Equations and conditions for obtaining the asymptotic solution (\ref%
{asymp-sol}) are derived by substitution of $x_{0}^{\mathrm{as}%
}(t,\varepsilon)$ and $y_{0}^{\mathrm{as}}(t,\varepsilon)$ into the problem (%
\ref{eq-x-closed})-(\ref{eq-y-closed}) instead of $x(t)$ and $y(t)$,
respectively, and equating the coefficients for the same powers of $%
\varepsilon$ on both sides of the resulting equations, separately for the
outer solution terms and the boundary correction terms.

\subsubsection{Obtaining $x_{0}^{b}(\protect\tau)$}

For obtaining this term, we derive the equation
\begin{equation}  \label{eq-x_0^b}
\frac{dx_{0}^{b}(\tau)}{d\tau}=0,\ \ \ \ \tau\ge 0.
\end{equation}
Due to the Boundary Function Method, we require that $x_{0}^{b}(\tau)%
\rightarrow 0$ for $\tau\rightarrow +\infty$. Subject to this requirement,
the equation (\ref{eq-x_0^b}) yields the unique solution
\begin{equation}  \label{x_0^b}
x_{0}^{b}(\tau)\equiv 0,\ \ \ \ \tau\ge 0.
\end{equation}

\subsubsection{Obtaining the outer solution}

Using the equations (\ref{eqL-1})-(\ref{eqH1H2}), (\ref{tilde-P_320}), (\ref%
{h_20}) and (\ref{mathcal-A_10})-(\ref{mathcal-A_40}), we have the system
for $\{x_{0}^{o}(t) , y_{0}^{o}(t)\}$
\begin{equation}  \label{eq-outer-x}
\frac{dx_{0}^{o}(t)}{dt}=\big(A_{1}-S_{1}P_{1 0}^{*}\big)%
x_{0}^{o}(t)+A_{2}y_{0}^{o}(t)-S_{1}h_{1 0}(t)+f_{1}(t),
\end{equation}
\begin{equation}  \label{eq-outer-y}
0=-\big(D_{2}\big)^{-1/2}A_{2}^{T}P_{1 0}^{*}x_{0}^{o}(t)-\big(D_{2}\big)%
^{1/2}y_{0}^{o}(t)-\big(D_{2}\big)^{-1/2}A_{2}^{T}h_{1 0}(t).
\end{equation}
Solving the equation (\ref{eq-outer-y}) with respect to $y_{0}^{o}(t)$, we
obtain
\begin{equation}  \label{y_0^o}
y_{0}^{o}(t)=-D_{2}^{-1}A_{2}^{T}P_{1
0}^{*}x_{0}^{o}(t)-D_{2}^{-1}A_{2}^{T}h_{1 0}(t).
\end{equation}
Then, substituting (\ref{y_0^o}) into (\ref{eq-outer-x}), and using (\ref%
{S_1^o}) and (\ref{matr-tildA_10}) yields the differential equation with
respect to $x_{0}^{o}(t)$
\begin{equation}  \label{eq-outer-x-2}
\frac{dx_{0}^{o}(t)}{dt}={\mathcal{A}}_{0}x_{0}^{o}(t)-S_{0}h_{1
0}(t)+f_{1}(t).
\end{equation}
Moreover, using (\ref{x_0^b}), we directly have the initial condition for
this equation
\begin{equation}  \label{in-cond-x_0^o}
x_{0}^{o}(0)=x_{0}.
\end{equation}
The solution of the problem (\ref{eq-outer-x-2})-(\ref{in-cond-x_0^o}) is
\begin{eqnarray}  \label{x_0^o}
x_{0}^{o}(t)=\exp\big({\mathcal{A}}_{0}t\big)x_{0}  \notag \\
-\int_{0}^{t}\exp\big({\mathcal{A}}_{0}(t-\sigma)\big)\Big(S_{0}h_{1
0}(\sigma) -f_{1}(\sigma)\Big)d\sigma,\ \ \ \ t\ge 0.
\end{eqnarray}
Due to (\ref{f-ineq}), (\ref{eigenval-ineq-2}), (\ref{ineq-h_10}) and (\ref%
{mu}), this solution satisfies the inequality
\begin{equation}  \label{x_0^O-ineq}
\big\|x_{0}^{o}(t)\big\|\le a\exp(-\mu t),\ \ \ \ t\ge 0,
\end{equation}
where $a>0$ is some constant.

The equations (\ref{mu}), (\ref{y_0^o}), and the inequalities (\ref%
{ineq-h_10}), (\ref{x_0^O-ineq}) yield
\begin{equation}  \label{y_0^O-ineq}
\big\|y_{0}^{o}(t)\big\|\le a\exp(-\mu t),\ \ \ \ t\ge 0,
\end{equation}
where $a>0$ is some constant.

\subsubsection{Obtaining $y_{0}^{b}(\protect\tau)$}

Using (\ref{eqL-1})-(\ref{eqH1H2}), (\ref{tilde-P_320}), (\ref{mathcal-A_40}%
), we have the following equation for this boundary correction term:
\begin{equation}  \label{eq-y_0^b}
\frac{dy_{0}^{b}(\tau)}{d\tau}=-\big(D_{2}\big)^{1/2}y_{0}^{b}(\tau),\ \ \ \
\tau\ge 0.
\end{equation}
Moreover, the initial condition for this equation is
\begin{equation}  \label{in-cond-y_0^b}
y_{0}^{b}(0)=y_{0}-y_{0}^{o}(0).
\end{equation}
The solution of the problem (\ref{eq-y_0^b})-(\ref{in-cond-y_0^b}) has the
form
\begin{equation}  \label{y_0^b}
y_{0}^{b}(\tau)=\exp\Big(-\big(D_{2}\big)^{1/2}t\Big)\big(y_{0}-y_{0}^{o}(0)%
\big),\ \ \ \ \tau\ge 0.
\end{equation}
Due to (\ref{eigenval-ineq-1}), this solution satisfies the inequality
\begin{equation}  \label{y_0^b-ineq}
\big\|y_{0}^{b}(\tau)\big\|\le a\exp(-\beta\tau),\ \ \ \ \tau\ge 0,
\end{equation}
where $a>0$ is some constant.

Thus, we have completed the formal construction of the zero-order asymptotic
solution to the problem (\ref{eq-x-closed})-(\ref{eq-y-closed}).

\begin{lemma}
\label{justif-asymp-x^as-y^as}Let the assumptions (A1)-(A7) be valid. Then,
there exists a positive constant $\tilde{\varepsilon}_{1}$, ($\tilde{%
\varepsilon}_{1}\le\varepsilon_{2}$), such that for all $\varepsilon\in(0,%
\tilde{\varepsilon}_{1}]$ the solution $\{x(t,\varepsilon) ,
y(t,\varepsilon)\}$ of the initial-value problem (\ref{eq-x-closed})-(\ref%
{eq-y-closed}) satisfies the inequalities
\begin{equation}  \label{asymp-x-ineq}
\big\|x(t,\varepsilon)-x_{0}^{o}(t)\big\|\le c\varepsilon\exp(-\mu t),\ \ \
\ \ t\ge 0,
\end{equation}
\begin{equation}  \label{asymp-y-ineq}
\big\|y(t,\varepsilon)-y_{0}^{o}(t)-y_{0}^{b}(t/\varepsilon)\big\|\le
c\varepsilon\exp(-\mu t),\ \ t\ge 0,
\end{equation}
where $c>0$ is some constant independent of $\varepsilon$.
\end{lemma}

\begin{proof}
The lemma is proven similarly to Lemma \ref{asympt-h}.
\end{proof}

Let us denote
\begin{eqnarray}  \label{J_0}
J_{0}\overset{\triangle}{=}\int_{0}^{+\infty}\Big[\big(x_{0}^{o}(t)\big)%
^{T}D_{1}x_{0}^{o}(t)  \notag \\
+\Big(P_{1 0}^{*}x_{0}^{o}(t)+h_{1 0}(t)\Big)^{T}S_{0}\Big(P_{1
0}^{*}x_{0}^{o}(t)+h_{1 0}(t)\Big)\Big]dt.
\end{eqnarray}

\begin{corollary}
\label{justif-asymp-J}Let the assumptions (A1)-(A7) be valid. Then, for all $%
\varepsilon\in(0,\tilde{\varepsilon}_{1}]$ the following inequality is
satisfied:
\begin{equation}  \label{ineq-J-asympt}
\big|J\big(u_{\varepsilon,1}(z,t)\big)-J_{0}\big|\le c\varepsilon.
\end{equation}
\end{corollary}

\begin{proof}
Substitution of $z(t,\varepsilon)=\mathrm{col}\big(x(t,\varepsilon) ,
x(t,\varepsilon)\big)$ and $u_{\varepsilon,1}\big(z(t,\varepsilon),t\big)$
(see (\ref{derivation})) into (\ref{new-perf-ind}), and using (\ref%
{mathcal-G_u}) and (\ref{new-matr-D}) yields
\begin{eqnarray}  \label{J(u_eps,1)}
J\big(u_{\varepsilon,1}(z,t)\big)= \int_{0}^{+\infty}\Big[%
x^{T}(t,\varepsilon)D_{1}x(t,\varepsilon)+y^{T}(t,\varepsilon)D_{2}y(t,%
\varepsilon)  \notag \\
+ \big(K_{1}(\varepsilon)x+\varepsilon K_{2}(\varepsilon)y +H_{3}h_{1
0}(t)+\varepsilon H_{1}h_{2 0}(t)\big)^{T}\widetilde{G}  \notag \\
\times\big(K_{1}(\varepsilon)x+\varepsilon K_{2}(\varepsilon)y +H_{3}h_{1
0}(t)+\varepsilon H_{1}h_{2 0}(t)\big)\Big]dt,
\end{eqnarray}
where $\widetilde{G}$ is given in (\ref{eqH1H2}).

Now, using the equations (\ref{K}), Lemma \ref{justif-asymp-x^as-y^as} and
the inequalities (\ref{x_0^O-ineq}), (\ref{y_0^O-ineq}), (\ref{y_0^b-ineq}),
we can represent the expression (\ref{J(u_eps,1)}) as follows:
\begin{eqnarray}  \label{J(u_eps,1)-2}
J\big(u_{\varepsilon,1}(z,t)\big)= \int_{0}^{+\infty}\Big[\big(x_{0}^{o}(t)%
\big)^{T}D_{1}x_{0}^{o}(t) + \big(y_{0}^{o}(t)\big)^{T}D_{2}y_{0}^{o}(t)
\notag \\
+ \big(P_{1 0}x_{0}^{o}(t) + h_{1 0}(t)\big)^{T}H_{3}^{T}\widetilde{G}H_{3}%
\big(P_{1 0}x_{0}^{o}(t) + h_{1 0}(t)\big)\big]dt + l(\varepsilon),
\end{eqnarray}
where $l(\varepsilon)$ is some function of $\varepsilon$ satisfying the
inequality
\begin{equation}  \label{ineq-l}
\|l(\varepsilon)\|\le c\varepsilon,\ \ \ \ \varepsilon\in(0,\tilde{%
\varepsilon}_{1}],
\end{equation}
$c>0$ is some constant independent of $\varepsilon$.

Now, using (\ref{y_0^o}), we have
\begin{equation}
\big(y_{0}^{o}(t)\big)^{T}D_{2}y_{0}^{o}(t)=\big(P_{10}x_{0}^{o}(t)+h_{10}(t)%
\big)^{T}A_{2}D_{2}^{-1}A_{2}^{T}\big(P_{10}x_{0}^{o}(t)+h_{10}(t)\big).
\label{expres-1}
\end{equation}%
Also, using (\ref{eqL-1})-(\ref{eqH1H2}), we obtain
\begin{equation}
H_{3}^{T}\widetilde{G}H_{3}=S_{1}.  \label{expres-2}
\end{equation}%
Finally, the substitution of (\ref{expres-1})-(\ref{expres-2}) into (\ref%
{J(u_eps,1)-2}), and using the equation (\ref{S_1^o}) and the inequality (%
\ref{ineq-l}) lead immediately to the inequality (\ref{asymp-y-ineq}). Thus,
the corollary is proven.
\end{proof}

\subsection{Main part of the proof}

Substitution of the optimal control (\ref{bar-u*}) of the ROCP into the
dynamics (\ref{out-dyn1}) of this problem, and using the equations (\ref%
{S_1^o-new-repr}), (\ref{matr-tildA_10}) yields after some rearrangement the
following initial-value problem for the optimal trajectory $\bar{x}^{\ast
}(t)$, $t\geq 0$ in the ROCP:
\begin{equation}
\frac{d\bar{x}^{\ast }(t)}{dt}={\mathcal{A}}_{0}\bar{x}^{\ast
}(t)-S_{0}h_{10}(t)+f_{1}(t),\ \ \ \ \bar{x}^{\ast }(0)=x_{0}.
\label{opt-traj-ROCP}
\end{equation}%
Comparison of this problem with the problem (\ref{eq-outer-x-2})-(\ref%
{in-cond-x_0^o}) yields
\begin{equation}
\bar{x}^{\ast }(t)\equiv x_{0}^{o}(t),\ \ \ \ t\geq 0.  \label{bar-x-x^o}
\end{equation}%
Replacing $\bar{x}$ with $x_{0}^{o}(t)$ in (\ref{bar-u*}), we obtain the
time realization $\bar{u}^{\ast }(t)$ of the optimal state-feedback control
in the ROCP
\begin{equation}
\bar{u}^{\ast }(t)=-\Theta ^{-1}\bar{B}^{T}P_{10}^{\ast }x_{0}^{o}(t)-\Theta
^{-1}\bar{B}^{T}h_{10}(t),\ \ \ \ t\geq 0.  \label{time-real-ROCP}
\end{equation}%
Now, substituting (\ref{bar-x-x^o}) and (\ref{time-real-ROCP}) into (\ref%
{bar-J}) instead of $\bar{x}(t)$ and $\bar{u}(t)$, respectively, and using (%
\ref{S_1^o-new-repr}) yield the equality
\begin{equation}
\bar{J}^{\ast }=J_{0},  \label{bar-J=J_0}
\end{equation}%
where $\bar{J}^{\ast }$ is the optimal value of the cost functional in the
ROCP, while the value $J_{0}$ is given by (\ref{J_0}). Finally, the
equalities (\ref{J*=barJ}), (\ref{bar-J=J_0}) and the inequality (\ref%
{ineq-J-asympt}) directly imply the statement of the theorem.


\begin{thebibliography}{99}
\bibitem{pontr-bolt-gakr-misch-86} Pontriagin, L.S., Boltyanskii, V.G.,
Gamkrelidze, R.V., Mischenko, E.F.: The Mathematical Theory of Optimal
Processes, Gordon $\&$ Breach, New York (1986)

\bibitem{bel-57} Bellman, R.: Dynamic Programming, Princeton University
Press, Princeton, NJ (1957)

\bibitem{kelly-64} Kelly, H. J.: A second variation test for singular
extremals. AIAA Journal 2, 26--29 (1964)

\bibitem{bel-jac-75} Bell, D.J., Jacobson, D.H.: Singular Optimal Control
Problems, Academic Press, New York (1975)

\bibitem{gab-kir-72} Gabasov, R., Kirillova, F.M.: High order necessary
conditions for optimality. SIAM J. Control 10, 127--168 (1972)

\bibitem{Mehr-89} Mehrmann, V.: Existence, uniqueness, and stability of
solutions to singular linear quadratic optimal control problems. Linear
Algebra Appl. 121, 291--331 (1989)

\bibitem{krotov-96} Krotov, V.F.: Global Methods in Optimal Control Theory,
Marsel Dekker, New York (1996)

\bibitem{Ferran-Ntog-2014} Ferrante, A., Ntogramatzidis, L,: Continuous-time
singular linear-quadratic control: necessary and sufficient conditions for
the existence of regular solutions. \emph{arXiv:1404.1667v1 [math.OC]}, 12 p
(2014)

\bibitem{gur-65} Gurman, V.I.: Optimal processes of singular control. Autom.
Remote Control 26, 783---792 (1965)

\bibitem{gur-dix-77} Gurman, V.I., Dykhta, V.A.: Singular problems of
optimal control and the method of multiple maxima. Autom. Remote Control 38,
343---350 (1977)

\bibitem{gur-kang-11-1} Gurman, V.I., Ni Ming Kang: Degenerate problems of
optimal control. I. Autom. Remote Control 72, 497--511 (2011)

\bibitem{gur-kang-11-2} Gurman, V.I., Ni Ming Kang: Degenerate problems of
optimal control. II. Autom. Remote Control 72, 727---739 (2011)

\bibitem{gur-kang-11-3} Gurman, V.I., Ni Ming Kang: Degenerate problems of
optimal control. III. Autom. Remote Control 72, 929----943 (2011)

\bibitem{Hau-Silv-83} Hautus, M.L.J., Silverman, L.M.: System structure and
singular control. Linear Algebra Appl. 50, 369--402 (1983)

\bibitem{Wil-Kit-Sil-86} Willems, J.C., Kitapci, A., Silverman, L.M.:
Singular optimal oontrol: a geometric approach. SIAM J. Control Optim. 24,
323---337 (1986)

\bibitem{Geerts-89} Geerts, T.: All optimal controls for the singular
linear-quadratic problem without stability; a new interpretation of the
optimial cost. Linear Algebra Appl. 116, 135-181 (1989)

\bibitem{Geerts-94} Geerts, T.: Linear-quadratic control with and without
stability subject to general implicit continuous-time systems:
coordinate-free interpretations of the optimal costs in terms of dissipation
inequality and linear matrix inequality; existence and uniqueness of optimal
controls and state trajectories. Linear Algebra Appl. 203-204, 607-658 (1994)

\bibitem{zav-ses-97} Zavalishchin, S.T., Sesekin, A.N.: Dynamic Impulse
Systems: Theory and Applications, Kluwer Academic Publishers, Dordrecht
(1997)

\bibitem{Glizer2012a} Glizer, V.Y.: Solution of a singular optimal control
problem with state delays: a cheap control approach. In: Reich, S.,
Zaslavski, A.J. (eds.): Optimization Theory and Related Topics, Contemporary
Mathematics Series, vol. 568, pp. 77--107. American Mathematical Society,
Providence, RI (2012)

\bibitem{Glizer2012b} Glizer, V.Y.: Stochastic singular optimal control
problem with state delays: regularization, singular perturbation, and
minimizing sequence. SIAM J. Control Optim. 50, 2862--2888 (2012)

\bibitem{Glizer2014} Glizer, V.Y. Singular solution of an infinite horizon
linear-quadratic optimal control problem with state delays. In: Wolansky,
G., Zaslavski, A.J. (eds.): Variational and Optimal Control Problems on
Unbounded Domains, Contemporary Mathematics Series, vol. 619, pp. 59--98.
American Mathematical Society, Providence, RI (2014)

\bibitem{Tikh-Ars-77} Tikhonov, A.N., Arsenin, V.Y.: Solutions of Ill-Posed
Problems, Halsted Press, New York (1977)

\bibitem{GlizerFridmanTuretsky2007} Glizer, V.Y., Fridman, L.M., Turetsky,
V.: Cheap suboptimal control of an integral sliding mode for uncertain
systems with state delays. IEEE Trans. Automat. Control, 52, 1892--1898
(2007)

\bibitem{Glizer-Kelis2015} Glizer, V.Y., Kelis, O.: Solution of a zero-sum
linear quadratic differential game with singular control cost of minimizer.
Journal of Control and Decision, 2, 155--184 (2015)

\bibitem{omal-jam-77} O'Malley, R.E., Jameson, A.: Singular perturbations
and singular arcs, II. IEEE Trans. Automat. Control 22, 328--337 (1977)

\bibitem{sab-san-87} Sabery, A., Sannuti, P.: Cheap and singular controls
for linear quadratic regulators. IEEE Trans. Automat. Control 32, 208--219
(1987)

\bibitem{ser-br-kok-may-99} Seron, M.M., Braslavsky, J.H., Kokotovic, P.V.,
Mayne, D.Q.: Feedback limitations in nonlinear systems: from Bode integrals
to cheap control. IEEE Trans. Automat. Control 44, 829-833 (1999)

\bibitem{Glizer1999} Glizer, V.Y.: Asymptotic solution of a cheap control
problem with state delay. Dynam. Control, 9, 339--357 (1999)

\bibitem{sme-sob-05} Smetannikova, E.N., Sobolev, V.A.: Regularization of
cheap periodic control problems. Automat. Remote Control 66, 903--916 (2005)

\bibitem{Glizer2009} Glizer, V.Y.: Infinite horizon cheap control problem
for a class of systems with state delays. J. Nonlinear Convex Anal. 10,
199--233 (2009)

\bibitem{OReil83} O'Reilly, J.: Partial cheap control of the time-invariant
regulator. Internat. J. Control 37, 909--927 (1983)

\bibitem{Salukvadze-62} Salukvadze, M.E.: The analytical design of an
optimal control in the case of constantly acting disturbances. Automat.
Remote Control 23, 657--667 (1962)

\bibitem{Anders-Moore1971} Anderson, B.O.D, Moore, J.B.: Linear Optimal
Control, Prentice-Hall, Englewood, NJ (1971)

\bibitem{Kokotovic} Kokotovic, P.V., Khalil, H.K., O' Reilly, J.: Singular
Perturbation Methods in Control: Analysis and Design, Academic Press,
London, UK (1986)

\bibitem{Gl-2003} Glizer, V.Y.: Blockwise estimate of the fundamental matrix
of linear singularly perturbed differential systems with small delay and its
application to uniform asymptotic solution. J. Math. Anal. Appl. 278,
409-433 (2003)

\bibitem{Vasil'eva} Vasil'eva, A.B., Butuzov, V.F., Kalachev, L.V.: The
Boundary Function Method for Singular Perturbation Problems, SIAM Books,
Philadelphia, PA: (1995)
\end{thebibliography}
\end{document}